\documentclass[12pt]{article}
\usepackage{amsmath,amsfonts, amssymb, wasysym, graphicx}
\usepackage{hyperref}
\hypersetup{colorlinks,citecolor=blue,linkcolor=blue}

\title{Piecewise Visual, Linearly Connected  Metrics on Boundaries of Relatively Hyperbolic Groups} 

\author{M. Haulmark and M. Mihalik}

\usepackage{tikz}
\usepackage{enumerate}
\newtheorem{theorem}{Theorem}[section]
\newtheorem{definition}[theorem]{Definition}
\newtheorem{proposition}[theorem]{Proposition}
\newtheorem{lemma}[theorem]{Lemma}
\newtheorem{claim}[theorem]{Claim}
\newtheorem{corollary}[theorem]{Corollary}
\newtheorem{remark}[theorem]{Remark}
\newtheorem{example}[theorem]{Example}
\newtheorem{question}[theorem]{Question}
\newenvironment{proof}{\addvspace{12pt}\noindent{\bf Proof:}}{
$\Box$\par\addvspace{12pt}}

\date{\today}
\begin{document}
\maketitle

\begin{abstract} 
Suppose a finitely generated group $G$ is hyperbolic relative to $\mathcal P$ a set of proper finitely generated subgroups of $G$. Established results in the literature imply that a ``visual" metric on $\partial (G,\mathcal P)$ is ``linearly connected" if and only if the boundary $\partial (G,\mathcal P)$ has no cut point. Our goal is to produce linearly connected metrics  on $\partial (G,\mathcal P)$ that are ``piecewise" visual when $\partial (G,\mathcal P)$ contains cut points. 

Our main theorem is about graph of groups decompositions of relatively hyperbolic groups $(G,\mathcal P)$, and piecewise visual metrics on their boundaries.  We assume that each vertex group of our decomposition has a boundary with linearly connected visual metric or the vertex group is in $\mathcal P$. If a vertex group is not in $\mathcal P$, then it is hyperbolic relative to its adjacent edge groups. Our linearly connected metric on $\partial (G,\mathcal P)$ agrees with the visual metric on limit sets of vertex groups and is in this sense piecewise visual.


\end{abstract}

\section{Introduction}\label{Intro} 

The following technical result is our main theorem.

\begin{theorem}\label{main} 
Suppose:

1) The finitely generated group $G$ is hyperbolic relative to a finite collection $\mathcal P$ of finitely generated subgroups, the boundary $\partial (G,\mathcal P)$ is connected,  and $G$ has a graph of groups decomposition $\mathcal A$, where 
each vertex and edge group of $\mathcal A$ is finitely generated and infinite.

2) Each element of $\mathcal P$ is either a vertex or edge group of $\mathcal A$, 

3) If $V_i$ is a vertex group of $\mathcal A$,  $V_i\not\in \mathcal P$ and $\mathcal P_i$ is the set of edge groups of $\mathcal A$ adjacent to $V_i$ then $V_i$ is hyperbolic relative to $\mathcal P_i$,  $\partial (V_i,\mathcal P_i)$ is connected, locally connected and has no cut point and each member of $\mathcal P_i$ is a subgroup of a member of $\mathcal P$.  

Then given a visual metric $d_V$ on the topological space $\partial (G, \mathcal P)$ there is a ``piecewise visual", linearly connected metric $d_L$ on $\partial (G,\mathcal P)$ such that if $x_1$ and $x_2$ are points in  the limit set of $gV_i$ ($g\in G$ and $V_i$ a vertex group of $\mathcal A$), then $d_L(x_1,x_2)=d_V(x_1,x_2)$.
\end{theorem} 

\begin{corollary}\label{mainCor} 
Suppose $(G,\mathcal P)$ is relatively hyperbolic and $\partial (G,P)$ is connected, locally connected and all cut points are parabolic. If all edge groups in the maximal peripheral  splitting of $(G,\mathcal P)$ (Theorem \ref{Acc}) are finitely generated, then there is a piecewise visual linearly connected metric on $\partial (G,\mathcal P).$  
\end{corollary}

Some comments about our hypotheses are in order. In hypotheses $1)$ and $3)$ of the theorem we assume $\partial (G,\mathcal P)$ and $\partial (V_i,\mathcal P_i)$ are connected. If $G$ is 1-ended, then certainly $\partial (G,\mathcal P)$ is connected, but $G$ need not be 1-ended in order for $\partial (G,\mathcal P)$ to be connected. If $G$ is the free group on $\{x,y\}$ and $P$ is the infinite cyclic group generated by the commutator $xyx^{-1}y^{-1}$ then $\partial (G,P)$ is homeomorphic to a circle. Hypothesis $3)$ requires vertex group boundaries to be connected and locally connected. There is no known example of a relatively hyperbolic group with boundary that is connected and not locally connected. 

Our proof of this theorem is carried out in a cusped space $X$ for $(G,\mathcal P)$. The space $X$ is hyperbolic and the boundary of $X$ is $\partial (G,\mathcal P)$. The space $X$ is built from a Cayley graph $\Gamma$ of $G$ (see \S \ref{Cusped}). Since the vertices of $\Gamma$ are the elements of $G$ the limit set of $gV_i$ (referred to in the theorem) is a subset set of $\partial (G,\mathcal P)$.

If a space has a linearly connected metric, then it is locally connected, but even the unit interval with usual topology has metrics which are not linearly connected. Bonk and B. Kleiner \cite{BK05} prove that visual metrics on boundaries of 1-ended hyperbolic groups are linearly connected. J. Mackay and A. Sisto \cite{McS18} prove that if $(G,\mathcal P)$ is a relatively hyperbolic pair and $\partial (G,\mathcal P)$ is connected, locally connected and without cut point, then any visual metric on $\partial (G,\mathcal P)$ is linearly connected. If $\partial (G,\mathcal P)$ has a cut point, then any visual metric on this space is not linearly connected (see \cite{GHM19}). Our goal here is to consider connected boundaries of relatively hyperbolic groups and produce ``piecewise visual" linearly connected metrics on these boundaries (in the presence of cut points). If $X$ is a cusped space for the relatively hyperbolic pair $(G,\mathcal P)$, $d_V$ is a visual metric on $\partial X =\partial (G, \mathcal P)$, then our hypotheses imply the relatively hyperbolic vertex groups of our decomposition have linearly connected boundary. We show the limit set of any coset of any vertex group of our decomposition is linearly connected under $d_V$ and we define our proposed linearly connected metric on $\partial X$ to agree with $d_V$ on each such limit set.
Given any two points $x,y\in\partial X$, let $C_{(x,y)}=\{\ldots , c_{-1},c_0,c_1\ldots \}$ be the set cut points in $\partial X$ separating $x$ and $y$. This set may be finite, infinite or bi-infinite and is ordered by the Bass-Serre tree of the splitting. Since $\{c_i, c_{i+1}\}$ is a subset of the (linearly connected) limit set of a vertex group coset, $d_V(c_i,c_{i+1})=d_L(c_i,c_{i+1})$. If $C_{(x,y)}$ is bi-infinite, we define $d_L(x,y)=\cdots +d_V(c_{-1},c_0)+d_V(c_0,c_1)+\cdots $ and extend $d_L$ to all of $\partial X$ in a similar way. We must show that the summations involved are convergent, $d_L$ is a well defined metric, $d_L$ and $d_V$ define the same topology and that $d_L$ is linearly connected. The most difficult of task is to show $d_L$ and $d_V$ define the same topology. We produce a constant $N$ and prove that if $d_V(x,y)<({\epsilon \over N})^4$ then $d_L(x,y)<\epsilon$ so that the identity function from the compact metric space $(\partial X, d_V)$ to the metric space $(\partial X, d_L)$ is continuous and hence a homeomorphism.

The paper is organized as follows: Our results connect with important splittings results for relatively hyperbolic groups. This is discussed in \S \ref{Bow} and Corollary \ref{mainCor} is proved at the end of this section. 
The basics of hyperbolic metrics, inner products  and visual metrics are covered in \S \ref{HI}. We examine inner products on the boundary of a hyperbolic space and show that ideal triangles are $\delta$ thin. 
In \S \ref{LCM} we define linearly connected metrics and show that $[0,1]$ with usual topology can be endowed with a non linearly connected metric. 
Basic definitions and results about cusped spaces and relatively hyperbolic groups are listed in \S \ref{Cusped}. Lemmas \ref{QC} and \ref{Deep} are fundamental to the proofs in the sections that follow this section. In order to prove our main theorem, we must know that the linearly connected visual metrics on our vertex groups carry over to linearly connected limit sets of their quasi-isometrically embedded images in the cusped space for the over group. This is a non-trivial matter since visual metrics are defined in terms of exponential functions. 
Section \ref{QIE} is devoted to a general result (Theorem \ref{sub}) that implies linear connectedness of boundaries is preserved by quasi-isometries. 
Our piecewise visual linearly connected metrics are defined in \ref{PCM}. Cut points in boundaries and separating subsets of our cusped space and how they separate geodesic lines are examined. Theorem \ref{conv} is the main result of this section. It concludes that our new distance function is a metric on the boundary of our cusped space. 
The most complex result of the paper is proved in \S \ref{dV=dL}. Theorem \ref{cont} shows that the visual metric and our linearly connected metric on the boundary of a relatively hyperbolic group (with cut points) generate the same topology. 
At this point, it is simply a matter of combining the results in the previous sections to prove our main theorem in Section \ref{Proof}. 
Finally in Section \ref{double} we ask if our piecewise visual metric is doubling, in the appropriate setting.

\section {Connections to Known Splittings} \label {Bow} 


\begin{definition} (\cite{Bow01}). Let $(G,\mathcal P)$ be a relatively hyperbolic group. A {\it peripheral splitting} of $G$ is a representation of $G$ as a finite bipartite graph of groups where $\mathcal P$ consists precisely of the (conjugacy classes of) vertex groups of one color. A peripheral splitting is a refinement of another if there is a color preserving folding of the first into the second.
\end{definition}

The hypotheses of our main theorem are similar to those in several of Bowditch's results and lead to a corollary. 
It is established in (\cite{Bow01}, Theorem 1.3) that if $\partial (G,\mathcal P)$ 
is connected, then any non-peripheral vertex group of a peripheral splitting also has connected boundary and is hyperbolic relative to its adjacent edge groups. The natural hyperbolic structure on vertex groups refereed to in the following accessibility result of Bowditch might not consist solely of adjacent edge groups. 

\begin{theorem} (\cite{Bow01}, Theorem 1.4) \label{Acc} 
Suppose the 1-ended group $G$ is hyperbolic relative to $\mathcal P$. Then $(G,\mathcal P)$ admits a (possibly trivial) maximal peripheral splitting. In other words,  $G$ splits over as a finite bipartite graph of groups $\mathcal G(G)$ with the following properties: Every $P\in\mathcal{P}$ is conjugate into a vertex group of one color, and each vertex group $H$ inherits a natural relatively hyperbolic structure $(H,\mathcal Q)$ such that $H$ does not split over a finite or parabolic subgroup relative to $\mathcal Q$.  
\end{theorem}

This splitting is called the {\it maximal peripheral splitting}. Recall that a splitting of $(G,\mathcal P)$ is {\it relative to $\mathcal P$} if each element of $\mathcal P$ is conjugate into a vertex group of the splitting. 

\begin{theorem} (\cite{Bow01}, Proposition 5.2) Suppose that $\Gamma$ is a group, and $\mathcal G$ is a peripheral structure with every peripheral subgroup 1-ended. If $\Gamma$  splits over a subgroup of a peripheral subgroup, then it splits relative to $\mathcal G$ over a subgroup of a peripheral subgroup.
\end{theorem}

\begin{proof} (of Corollary \ref{mainCor})
By Theorem \ref{Acc} $(G,\mathcal P)$ admits a maximal peripheral splitting $\mathcal G$ of $G$ with finitely generated vertex groups, and whose underlying graph is bipartite with vertices of one color corresponding to the of elements of $\mathcal{P}.$ By hypothesis the edge groups of $\mathcal G$ are finitely generated, and since $\partial (G,\mathcal{P})$ is connected $\mathcal{G}$ does not have any finite edge groups (See \cite{Bow01} Proposition 1.1). Thus $\mathcal G$ satisfies $(1)$ and $(2)$ of Theorem \ref{main}. By Theorem 1.3 of \cite{Bow01}, if $H$ is not a peripheral vertex, then $H$ is hyperbolic relative to $\mathcal{Q}$ the collection of edge groups adjacent to $H$. Since $\partial(G,\mathcal{P})$ is connected the limit set of $H$ is connected (\cite{Bow01} Theorem 1.3), moreover, this limit set is homeomorphic to $\partial(H,\mathcal{Q})$. Additionally, since $\partial(G,\mathcal{P})$ is locally connected and all cut points are parabolic the limit set of $H$ is locally connected (see \cite{Bow01} Propositions 7.4 and 8.5). Because $H$ does not admit a perihperal splitting, $\partial(H,\mathcal{Q})$ has no cut point (see Theorem 1.2 of \cite{Bow01}). Thus we have satisfied $(3)$ of Theorem \ref{main}.
\end{proof}

\section{Hyperbolicity and Inner Products} \label {HI}
\begin{definition}\label{IP}
If $X$ is a geodesic metric space with base point $p$, there is an inner product $``\cdot"$ for $X$ defined as follows:
If $x,y\in X$ define 
$$(x. y)_p={1\over 2}(d(p,x)+d(p,y)-d(x,y))$$ 
If there is a constant $\delta\geq 0$ such that for all $x,y,z\in X$: 
$$(x.p)_w\geq min\{(x.z)_p,(z.y)_p\}-\delta$$
we say that the inner product and the space $(X,d)$ are $\delta$-{\it hyperbolic}.
\end{definition}

There are a number of equivalent forms of hyperbolicity for geodesic metric spaces. In this paper we use the following {\it thin triangles} definition. 

\begin{definition} \label{hyp} 
Suppose $(X,d)$ is a geodesic metric space.  If $\triangle(x,y,z)$ is a geodesic triangle in $X$, let $\triangle '(x',y',z')$ be a Euclidean comparison triangle (i.e. $d'(x',y')=d(x,y)$ etc., where $d'$ is the Euclidean metric.) There is a homeomorphism $f:\triangle '\to \triangle $ which is an isometry on each side of $\triangle$. The maximum inscribed circle in $\triangle '$ meets the side $[x',y']$ (respectively $[x',z']$, $[y',z']$) in a point $c_x'$ (resp. $c_y'$, $c_z'$) such that

$$d(x',c_z')=d(x',c_y'), \ d(y',c_x')=d(y',c_z'), \ d(z',c_y')=d(z',c_x'). $$

Let $c_x=f^{-1}(c_x')$, $c_y=f^{-1}(c_y')$ and $c_z=f^{-1}(c_z')$. We call the points $c_x,c_y, c_z$ the {\it internal points} of $\triangle $.
There is a unique continuous function $t_{\triangle}:\triangle '\to T_\triangle$ of  $\triangle '$ onto a tripod $T_\triangle$, where $t_\triangle$ is an isometry on the edges of $\triangle '$ and $T_\triangle$ is a tree with one vertex $w$ of degree 3, and vertices $x'', y'', z''$ each of degree one, such that $d(w,z'')=d(z,c_y)=d(z,c_x)$ etc. (See Figure 1.) Let $f_\triangle$ be the composite map $f_\triangle \equiv t_\triangle \circ f:\triangle\to T_\triangle$.
We say that $\triangle (x,y,z)$ is $\delta-thin$ if fibers of $f_\triangle$ have diameter at most $\delta$ in $X$. In other words, for all $p,q$ in $\triangle$,
$$f_\triangle (p)=f_\triangle (q) \hbox{ implies } d_X(p,q)\leq \delta .$$
 
 We say that {\it triangles are thin} if there is a constant $\delta$ such that all geodesic triangles in $X$ are $\delta$-{\it thin}. In this case we say $X$ is $\delta$-{\it hyperbolic}. 
 \end{definition}
 
 \vspace {.5in}
\vbox to 3in{\vspace {-2in} \hspace {-1.95in}
\includegraphics[scale=1]{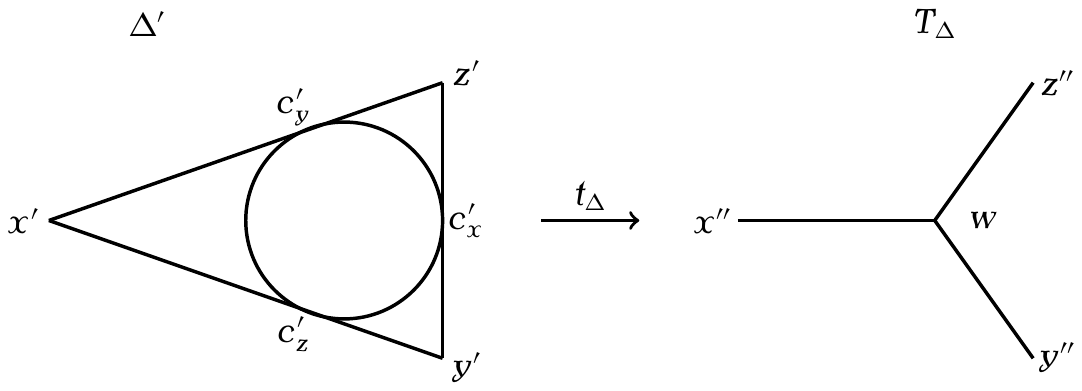}
\vss }
\vspace{-1.8in}

\centerline{Figure 1}

\medskip

\begin{lemma} \label{inpro} 
Suppose $(X,d)$ is a $\delta$-hyperbolic geodesic metric space and $\triangle(x,y,z)$ a geodesic triangle. If $c_x$, $c_y$ and $c_z$ are the internal points of $\triangle$, then $(y.z)_x=d(x,c_z)=d(x, c_y)$. 
\end{lemma}
\begin{proof} 
Notice that in the Euclidean comparison triangle
$$d(x',c_z')={1\over 2} (d(x',c_z')+d(x',c_y'))={1\over 2}(d(x',z')+d(x',y')-d(z',y')).$$
\end{proof}

\begin{lemma} \label{parallel} 
Let $(X,d)$ be a $\delta$-hyperbolic geodesic metric space.  Suppose $\alpha$ is a geodesic from $a_1$ to $a_2$, $\beta$ is a geodesic from $b_1$ to $b_2$ and  $K=max\{d(a_1,b_1),d(a_2,b_2),\delta\}$,  Then for any point $x$ on $\alpha$ there is a point $x'$ on $\beta$ such that $d(x,x')\leq K+2\delta$. Furthermore there are constants  $K_1,K_2,K_3\in [-K,K]$ such that $d(\alpha(K_1+i),\beta(K_2+i))\leq 2\delta$ for $0\leq i\leq |\alpha |-K_3-K_1$.
\end{lemma}
\begin{proof}
For $i\in\{1,2\}$ let $\gamma_i$ be a geodesic from $a_i$ to $b_i$ and $\tau$ a geodesic from $a_1$ to $b_2$. Consider the geodesic triangle $\triangle (\alpha,\gamma_2, \tau)$ with insize point $q_0$ on $\tau$, $q_1$ on $\gamma_2$ and $ q_2$ on $\alpha$. 
Consider the geodesic triangle $\triangle (\beta,\gamma_1, \tau)$ with insize point $p_0$ on $\tau$, $p_1$ on $\gamma_1$ and $ p_2$ on $\beta$ (see Figure 2). 

\vspace {.4in}
\vbox to 3in{\vspace {-1.8in} \hspace {-1.95in}
\includegraphics[scale=1]{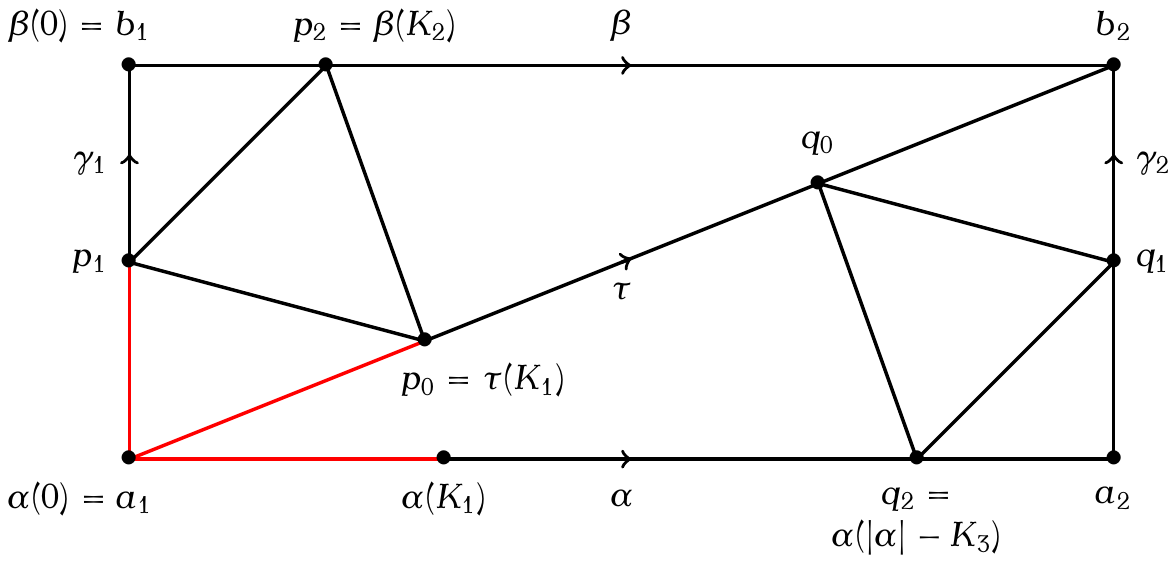}
\vss }
\vspace{-.8in}

\centerline{Figure 2}

\medskip

Let $d(a_1,p_0)= K_1$, $d(b_1,p_2)=K_2$  and $d(q_2, a_2)\leq K_3$. Notice that $K_1\leq K$, $K_2\leq \delta$ and $K_3\leq K$. 
If $t\in [0,K_1]$ or $[|\alpha |-K_3,|\alpha|]$ then $d(\alpha(t) ,\{b_1,b_2\})\leq K+\delta$. Otherwise $t=K_1+i$ and $d(\alpha(K_1+i), \beta (K_2+i))\leq 2\delta$  for $i+K_1\leq |\alpha|-K_3$. 
\end{proof}

If $X$ is a hyperbolic geodesic metric space, the points of $\partial X$ are equivalence classes $[r]$ of geodesic rays $r$, where $r$ and $s$ are equivalent if there is a number $K\geq 0$ such that $d(r(k),s(k))\leq K$ for all $k\geq 0$. Note that if such a $K$ exists for $r,s$ based at $p$, then our thin triangle condition forces $d(r(k),s(k))\leq \delta$ for all $k\geq 0$. (Simply consider the geodesic triangle formed by $r([0,k+K])$, $s([0,k+K])$ and a geodesic (of length $\leq K$) connecting $r(k+K)$ to $s(k+K)$. The internal points on $r$ and $s$ are beyond $r(k)$ and $s(k)$ respectively.)

The inner product extends to $X\cup \partial X$ (see Definition 4.4 \cite{ABC91}). Given a geodesic line $l:(-\infty,\infty)\to X$, we say $l^-$ converges to $r$ if there is a constant $K$ such that $d(r(t), l(-t))\leq K$ for all $t\geq 0$. Similarly for $l^+$. The line $l$ and the rays $r$ and $s$ (based at $p\in X$) forms an {\it ideal geodesic triangle} if $l^-$ converges to $r$ and $l^+$ converges to $s$. Note that if $r$ and $r'$ converge to the same boundary point, and $(r,s, l)$ form an ideal triangle (with $l$ a line) then $(r',s,l)$ forms an ideal geodesic triangle. The next lemma shows that ideal triangles of this type are $5\delta$ thin.

\begin{lemma}  \label{Yes} 
Let $r$ and $s$ be geodesic rays based at $\ast$ in the hyperbolic space $X$, representing distinct points $x,y\in\partial(X)$, respectively.  If $m=(x.y)_\ast$ and $\ell$ is any line with $\ell^-\in x$ and $\ell^{+}\in y,$ then there is a $z\in\ell$ such that $d(r(m),z)\leq 5\delta$ and $d(s(m),z)\leq 5\delta$. If $l$ is parametrized such that $l(0)=z$, then $d(l(-j), r(m+j))\leq 5\delta$ for all $j\geq 0$, and $d(l(j), s(m+j))\leq 5\delta$ for all $j\geq 0$. We call the points $r(m)$, $s(m)$ and $z$ internal points of the ideal geodesic triangle $\triangle (r, s, l)$ (even though $z$ may not be unique). 
\end{lemma}

\begin{remark}\label{sub1} 
If a space is $\delta$ hyperbolic, then it is $\delta'$ hyperbolic for any $\delta'\geq \delta$. In order to simplify the implementation of Lemma \ref{Yes}, we replace our hyperbolicity constant $\delta$ by $5\delta$. This simply means that in all applications of Lemma \ref{Yes} we will assume that $d(r(m-j), s(m-j)\leq \delta$ for all $j\geq 0$, $d(r(m),z)\leq \delta$, $d(s(m),z)\leq \delta$, $d(l(-j),r(m+j))\leq \delta$ for all $j\geq 0$ and $d(l(j), s(m+j))\leq \delta$ for all $j\geq 0$. In other words, ideal geodesic triangles of the type considered here are $\delta$-thin.
\end{remark}
\begin{proof}
(of Lemma \ref{Yes}) For every $n\in \mathbb{N}$ let $\alpha_n$ be a geodesic in $X$ with endpoints $r(n)$ and $s(n)$. For every $n$ we define $a_n, u_n,$ and $v_n$ to be the insize points of $\Delta(\ast, r(n),s(n))$ with $a_n\in\alpha_n$, $u_n\in r,$ and $v_n\in s.$ Let $m=(x.y)_{\ast}.$ There is $T>0$ such that for all $t\geq T$, $d(r(t),s(t))>\delta$. We have $d(u_n,\ast)=d(v_n,\ast)\leq T.$ For every $n\in\mathbb{N}$,  $\{a_n,u_n,v_n\}\subset\overline{B}=\overline{B}(\ast,T+\delta).$ 

There are only finitely many vertices in $\overline{B},$ so there is $a\in  \overline{B}$ and a subsequence $\mathcal S_0$ of $(1,2,\ldots)$ such that  $a_n=a\in \overline{B}$ for all $n\in \mathcal S_0$. Passing to subsequences twice more we have a subsequence $\mathcal S_1$ of $\mathcal S_0$ such that $u_n=u$ and $v_n= v$ for all $n\in \mathcal S_1$. Notice that $u$, $v$ and $a$ are the insize points of the geodesic triangle with sides $r|_{[0,n]}$, $s|_{[0,n]}$, $\alpha_n$ for all $n\in \mathcal S_1$. Since $X$ is locally finite, we may construct a line $\alpha$ with $\alpha^-\in x$ and $\alpha^+ \in y$ via an Arzel\`{a}-Ascoli argument. Simply define $\alpha(0)=v$. There is a subsequence $\mathcal S_2$ of $\mathcal S_1$ such that for all $n\in \mathcal S_2$, the vertex of $\alpha_n$ preceding $a$ is the same (call it $\alpha(-1)$) and the vertex of $\alpha_n$  following $a$ is the same (call it $\alpha(1)$). Similarly select a subsequence $\mathcal S_3$ of $\mathcal S_2$ to define $\alpha(-2)$ and $\alpha(2)$. Continuing this in way, define the consecutive vertices of the geodesic line $\alpha$. Notice that if $k=d(\ast, u) (=d(\ast,v))$, then for each $j>0$, $d(\alpha(-j), r(k+j))\leq \delta$ and $d(\alpha(j), s(k+j))\leq \delta$. 

By construction, $d(a,u)\leq\delta$ and $d(a,v)\leq\delta.$ If $\ell$ is any line with $\ell^-\in x$ and $\ell^{+}\in y$ then $\ell$ is contained in the $2\delta$--neighborhood of $\alpha,$ so there is a point $z$ on $\ell$ such that $d(z,a)\leq2\delta.$  Assume $\ell$ is parametrized such that $\ell (0)=z$. Then $d(\ell(j),\alpha(j))\leq 2\delta$ for all $j$. Thus $d(\ell(j),r(k-j))\leq 3\delta$ and $d(\ell(j),s(k+j))\leq 3\delta$ for all $j\geq 0$. In particular, $d(z,u)\leq 3\delta$ and $d(z,v)\leq 3\delta$.  By \cite{ABC91}, Lemma 4.6(4)) $m\leq k\leq m+2\delta$, so that $d(z,r(m))\leq 5\delta$ and $d(z,s(m))\leq 5\delta.$  Finally, $d(\ell(j), r(m-j))\leq 5\delta$ and $d(\ell(j),s(m+j))\leq 5\delta$ for $j\geq 0$.
\end{proof}

\begin{definition}\label{visualm}
Let $X$ be a hyperbolic space with base point $p$. A metric $d_p$ on $\partial X$ is
called a (hyperbolic) {\it visual} metric with {\it parameter} $a>1$ and base point $p$ if there exist constants $k_1, k_2 > 0$ such that $k_1a^{-(x. y)_p} \leq d_p(x,y) \leq k_2a^{-(x.y)_p}$
for all $x,y\in \partial X$.
\end{definition}

\begin{remark} \label{equivar} 
If a group $G$ acts by isometries on the hyperbolic space $X$ then for $x,y\in \partial X$, $(x.y)_p=(gx.gy)_{g(p)}$. In this sense, the inner product is $G$-equivariant on $\partial X$.   If $d_p$ is a visual metric on $\partial X$ so that $k_1a^{-(r. s)_p} \leq d_p(r,s) \leq k_2a^{-(r.s)_p}$ then  for $g\in G$, one can define  $d_{gp}(gx,gy)=d_p(x,y)$. Then $k_1a^{-(gr.gs)_{gp}} \leq d_{gp}(gr,gs) \leq k_2a^{-(gr.gs)_{gp}}$. In this way $d_p$ can be thought of as $G$-equivariant. 
In particular, if $E\subset \partial X$, then the diameter of $E$ with respect to $d_p$ is equal to the diameter of $gE$ with respect to $d_{gp}$. 
\end{remark}

We are interested in the situation where $X$ is a cusped space for a relatively hyperbolic group $(G,\mathcal P)$, $p$ is a vertex  of $X$, and $d_p$ is a visual metric on $\partial X$.  
Since inner products are $G$-equivariant, Proposition 2.26 and Theorem 2.27 of \cite{BS07} (S. Buyalo and V. Schroeder) insure the existence of visual metrics on $\partial X=\partial(G,\mathcal P)$ (via the notion of finite chains of inner products of geodesic rays based at $p$). 

\section{Linearly Connected Metrics}\label{LCM} 
\begin{definition}
A metric $d$ on a space $X$ is {\it linearly connected} if there is a constant $K$ such that for each $x,y\in X$ there is a path of diameter $\leq Kd(x,y)$ connecting $x$ and $y$. 
\end{definition}

If a metric on a space $X$ is linearly connected, then $X$ is locally connected. But even the unit interval $[0,1]$ can have a metric that is not linearly connected. 

\begin{example}\label{E1}
Consider the homeomorphism of $[0,1]\to \mathbb R$ defined by $f(x)=x\sin({1\over x})$ for $x\in (0,1]$ and $f(0)=0$. Let $X$ be the graph of $f$, with metric induced by the standard metric on $\mathbb R^2$. It is straightforward to see that with this metric, $X$  is not linearly connected. Consider the points $x_k={2\over (4k+1)\pi}$ and $y_k={2\over (4k+3)\pi}$. Note that $\sin({1\over x_k})=1$ and $\sin({1\over y_k})=-1$ for all integers $k$. By the triangle inequality (for $k>0$), the distance between $(x_{k+1}, f(x_{k+1}))$ and $ (x_k, f(x_k))$ is less than $2(x_k-x_{k+1})={16\over \pi (4k+1)(4k+5)}$. But any path between these two points must pass through $(y_k, f(y_k))$, and so has diameter  greater than $2f(x_{k+1})=2x_{k+1}={4\over (4k+5)\pi}$. 
\end{example}

\section{Cusped Spaces for Relatively Hyperbolic Groups}\label{Cusped}

D. Groves and J. Manning \cite{GMa08} investigate a locally finite space $X$ derived from a finitely generated group $G$ and a collection ${\bf P}$ of finitely generated subgroups. The following definitions are directly from \cite{GMa08}

\begin{definition}  Let $\Gamma$ be any 1-complex. The combinatorial {\it horoball} based on $\Gamma$,
denoted $\mathcal H(\Gamma)$, is the 2-complex formed as follows:

{\bf A)} $\mathcal H^{(0)} =\Gamma (0) \times (\{0\}\cup \mathbb N)$

{\bf B)} $\mathcal H^{(1)}$ contains the following three types of edges. The first two types are
called horizontal, and the last type is called vertical.

(B1) If $e$ is an edge of $\Gamma$ joining $v$ to $w$ then there is a corresponding edge
$\bar e$ connecting $(v, 0)$ to $(w, 0)$.

(B2) If $k > 0$ and $0 < d_{\Gamma}(v,w) \leq 2^k$, then there is a single edge connecting
$(v, k)$ to $(w, k)$.

(B3) If $k\geq 0$ and $v\in \Gamma^{(0)}$, there is an edge joining $(v,k)$ to $(v,k+1)$.

{\bf C)} $\mathcal H^{(2)}$ contains three kinds of 2-cells:

(C1) If $\gamma \subset  \mathcal  H^{(1)}$ is a circuit composed of three horizontal edges, then there
is a 2-cell (a horizontal triangle) attached along $\gamma$.

(C2) If $\gamma \subset \mathcal H^{(1)}$ is a circuit composed of two horizontal edges and two
vertical edges, then there is a 2-cell (a vertical square) attached along $\gamma$. 

(C3) If $\gamma\subset  \mathcal H^{(1)}$ is a circuit composed of three horizontal edges and two
vertical ones, then there is a 2-cell (a vertical pentagon) attached along $\gamma$, unless $\gamma$ is the boundary of the union of a vertical square and a horizontal triangle.
\end{definition}

\begin{definition} Let $\Gamma$ be a graph and $\mathcal H(\Gamma)$ the associated combinatorial horoball. Define a {\it depth function}
$$\mathcal D : \mathcal H(\Gamma) \to  [0, \infty)$$
which satisfies:

(1) $\mathcal D(x)=0$ if $x\in \Gamma$,

(2) $\mathcal D(x)=k$ if $x$ is a vertex $(v,k)$, and

(3) $\mathcal D$ restricts to an affine function on each 1-cell and on each 2-cell.
\end{definition}

\begin{definition} Let $\Gamma$ be a graph and $\mathcal H = \mathcal H(\Gamma)$ the associated combinatorial horoball. For $n \geq 0$, let $\mathcal H_n \subset \mathcal H$ be the full sub-graph with vertex set $\Gamma ^{(0)} \times \{0,\ldots ,N\}$, so that $\mathcal H_n=\mathcal D^{-1}[0,n]$.  Let $\mathcal H^n=\mathcal D^{-1}[n,\infty)$ and $\mathcal H(n)=\mathcal D^{-1}(n)$.   The set $\mathcal H(n)$ is often called a {\it horosphere} or {\it $n^{th}$ level horosphere}. The set $\mathcal H^m$ is called an {\it $m$-horoball}. 
\end{definition}

\begin{lemma} \label{GM3.10} 
(\cite{GMa08}, Lemma 3.10) Let $\mathcal H(\Gamma)$ be a combinatorial horoball. Suppose that $x, y \in \mathcal H(\Gamma)$ are distinct vertices. Then there is a geodesic $\gamma(x, y) = \gamma(y, x)$ between $x$ and $y$  which consists of at most two vertical segments and a single horizontal segment of length at most 3.

Moreover, any other geodesic between $x$ and $y$ is Hausdorff distance at most 4 from this geodesic.
\end{lemma}

\begin{definition} Let $G$ be a finitely generated group, let ${\bf P} = \{P_1, \ldots , P_n\}$ be a (finite) family of finitely generated subgroups of $G$, and let $S$ be a generating set for $G$ containing generators for each of the $P_i$.   For each $i\in \{1,\ldots ,n\}$, let $T_i$ be a left transversal for $P_i$ (i.e. a collection of representatives for left cosets of $P_i$ in $G$ which contains exactly one element of each left coset).

For each $i$, and each $t \in T_i$, let $\Gamma_{i,t}$  be the full subgraph of the Cayley graph $\Gamma (G,S)$ which contains $tP_i$. Each $\Gamma_{i,t}$ is isomorphic to the Cayley graph of $P_i$ with respect to the generators $P_i \cap  S$. Then define
$$X(G,{\bf P},S) =\Gamma (G,S)\cup (\cup \{\mathcal H(\Gamma_{i,t})^{(1)} |1\leq i\leq n,t\in T_i\}),$$
where the graphs $\Gamma_{i,t} \subset  \Gamma(G,S)$ and $\Gamma_{i,t} \subset  \mathcal H(\Gamma_{i,t})$ are identified in the obvious way.
\end{definition}
The space $X(G,{\bf P}, S)$ is called the {\it cusped} space for  $G$, ${\bf P}$ and $S$. 
The next result shows cusped spaces are fundamentally important spaces.  We prove our results in cusped spaces. 

\begin{theorem} \label{GM3.25} 
(\cite{GMa08}, Theorem 3.25)
Suppose that $G$ is a finitely generated group and ${\bf P}=\{P_1,\ldots, P_n\}$ is a finite collection of finitely generated subgroups of $G$. Let $S$ be a finite generating set for $G$ containing generating sets for the $P_i$.  A  cusped space $X(G,{\bf P},S)$ is hyperbolic if and only if  $G$ is hyperbolic with respect to ${\bf P}$.
\end{theorem}

Assume $G$ is finitely presented and hyperbolic with respect to the subgroups ${\bf P}=\{P_1,\ldots, P_n\}$ and $S$ is a finite generating set for $G$ containing generating sets for the $P_i$. For $g\in G$ and $i\in\{1,\ldots, n\}$ we call $gP_i$ a {\it peripheral coset} in a cusped space. 
The isometric action of $G$ on $\Gamma(G,S)$ extends to an isometric action of $G$ on $X(G, {\bf P},S)$. This action is depth preserving. 


\begin{lemma} \label{geo} (\cite{GMa08}, Lemma 3.26) 
If a cusped space $X$ is $\delta$-hyperbolic, then the $m$-horoballs of $X$ are convex for all $m\geq \delta$. In particular, If $H$ is a horoball in $X$, then $H^\delta$ is convex. Given two points $a,b\in H^\delta$, there is a geodesic connecting $a$ and $b$ of the form $(\alpha, \tau,\beta)$ where $\alpha$ and $\beta$ are vertical and $\tau$ has length $\leq 3$. 
\end{lemma} 



\begin{lemma}\label{tight}
(\cite{MS18},Lemma 5.1)
Suppose $t_1$ and $t_2$ are vertices of depth $\bar d\geq \delta$ in a horoball $H$ of $X$. Then for each $i\in \{1,2\}$, there is a geodesic $\gamma_i$  from $\ast$ to $t_i$ such that  $\gamma_i$ has the form $(\eta_i, \alpha_i,\tau_i, \beta_i)$, where the end point $x_i$ of $\eta_i$ is the first point of $\gamma_i$ in $H(\bar d)$, $\alpha_i$ and $\beta_i$ are vertical and of the same length in $H^{\bar d}$ and $\tau_i$ is horizontal of length $\leq 3$. Furthermore $d(x_1,x_2)\leq 2\delta +1$. 
\end{lemma}

Let $H$ be a horoball of $X$ and $z$ a closest point of $H(\delta)$ to $\ast$. Lemma \ref{tight} implies that if $\gamma$ is a geodesic from $\ast$ to a point of $H^\delta$, then the first point of $\gamma$ in $H(\delta)$ is within $2\delta +1$ of $z$. 

For the remainder of the section, $(G,{\bf P})$ is relatively hyperbolic with cusped space $X$ and $C\in {\bf P}$. 

\begin{lemma} \label{close} 
Let $g$ be an element of $G$ and $q$ a closest point of $gC$ to $\ast$. If $\psi$ is a geodesic from $\ast$ to $gC$ that meets $gC$ only in its terminal point, then the terminal point of $\psi$ is within $6\delta+4$ of $q$. 
\end{lemma}

\begin{proof}
Let $H$ be the horoball for $gC(=H(0))$ and $z$ the vertex of $H(q)$ of vertical distance $\delta$ from $q$.
Note that $z$ is a closest point of $H(\delta)$ to $\ast$. Let the end point of $\psi$ be $a$. Let $\lambda$ be a vertical geodesic from $a$ to $b\in H(\delta)$. Let $(\eta,\alpha, \tau, \beta)$ be a geodesic (as in Lemma \ref{tight}) from $\ast$ to $b$. Let the end point of $\eta$ be $c$. By Lemma \ref{tight}, $d(z, c)\leq 2\delta+1$. Since $d(a,b)=\delta$, it suffices to show $|\beta| (=|\alpha|)\leq \delta$. (See Figure 3)

\vspace {.4in}
\vbox to 3in{\vspace {-2in} \hspace {-1.95in}
\includegraphics[scale=1]{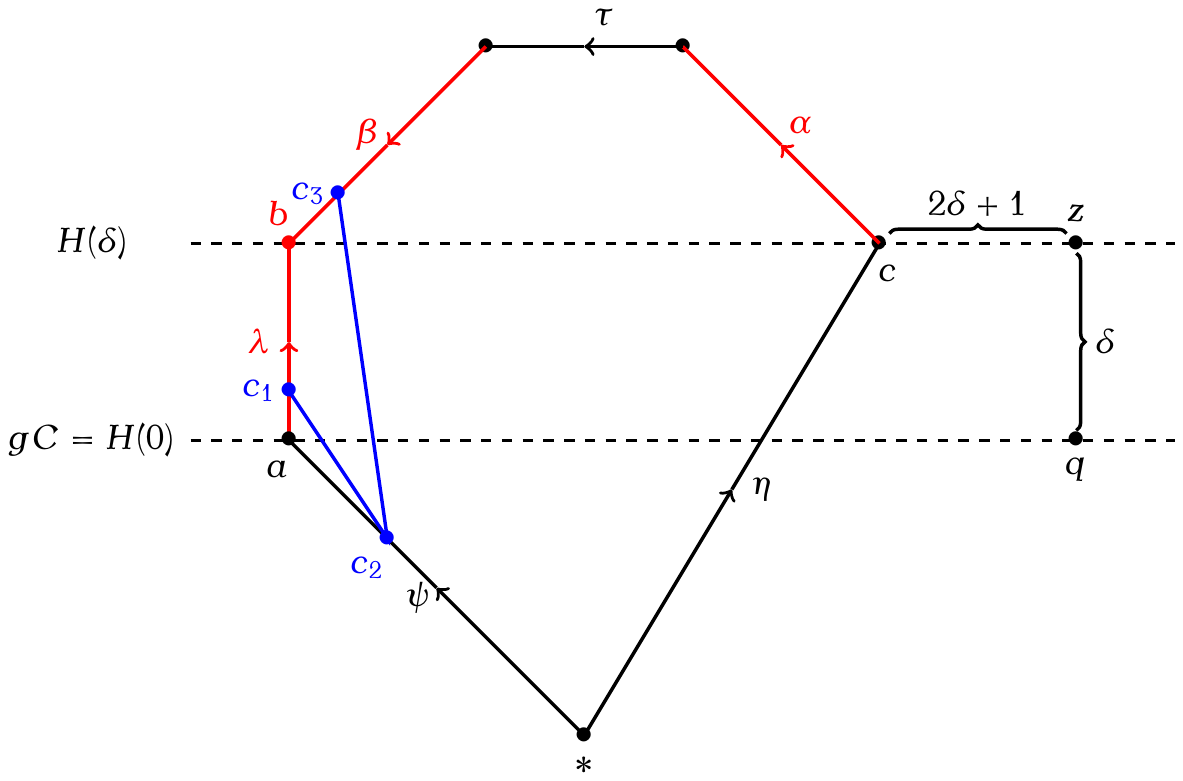}
\vss }
\vspace{-.1in}

\centerline{Figure 3}

\medskip

Consider the geodesic triangle with sides $\psi$, $\lambda$ and $(\eta, \alpha, \tau,\beta)$, with internal points $c_1$ on $\lambda$, $c_2$ on $\psi$ and $c_3$ on $(\eta, \alpha, \tau,\beta)$. If $|\beta|>\delta$, then the internal point $c_3$  must be on $\beta$ since $(\beta,\tau^{-1})$ is geodesic (vertical). But then the initial point of $\beta$ is within $\delta$ of a point of $\psi$, which is impossible.  Instead $|\beta|=|\alpha|\leq \delta$ and: 
$$d(q,a)\leq d(q,z)+d(z,c)+d(c,b)+d(b,a)\leq \delta+(2\delta+1)+(2\delta+3)+\delta.$$
\end{proof}

\begin{lemma} \label{QC} 
The horoballs of $X$ are quasi-convex. In fact, if $N\geq 0$ and $a_1$ and $a_2$ are vertices of $X$, both within $N$ of the horoball for $gC$  for some $g\in G$, then each point of a geodesic in $X$ connecting $a_1$ and $a_2$ is within $N+2\delta$ of a vertex of the horoball for $gC$. 
\end{lemma}
\begin{proof}
Let $H$ be the horoball for $gC$ and $\gamma_1'$ be a geodesic from $a_1$ to a closest point of $gC$. Then $|\gamma_1'|\leq N$ and the path $\gamma_1$ followed by a vertical geodesic to $H(\delta)$ is geodesic of length $\leq N+\delta$. Similarly consider a geodesic $\gamma_2$ from $a_2$ to a vertex of  $H(\delta)$, so that $|\gamma_2|\leq  N+\delta$. Let $b_1$ ($b_2$)  be the terminal point of $\gamma_1$ (respectively $\gamma_2$). 
Since $H^\delta$ is convex (Lemma \ref{geo}), Lemma \ref{parallel} implies every point of a geodesic $\alpha$ connecting $a_1$ and $a_2$ is 
within $N+3\delta$ of a point  of a geodesic connecting $b_1$ to $b_2$ (in $H^\delta$). Hence each point of $\alpha$ is within $N+2\delta$ of $H$. 
\end{proof}

In the next lemma we assume $\delta\geq 4$. 
\begin{lemma}\label{Deep} 
Suppose $N\geq 0$. If  $H$ is the horoball for $gC$ for some $g\in G$ and $\alpha:[0,k]\to X$ is a geodesic with $\alpha(0)$ and $\alpha(k)$ both within $N$ of $H$, then $\alpha$ maps the interval $[N+3  \delta,k-(N+3 \delta)]$ into $H^{\delta}$. Furthermore, if $\alpha(0)\in H(0)=gC$ then there is a constant $J_{\ref{Deep}} (N, \delta)$ such that if $|\alpha|\geq J_{\ref{Deep}}(N,\delta)$ then there is a geodesic $\beta$ from $\alpha(0)$ to $\alpha(k)$ such that an initial segment of $\beta$ is vertical of length ${k-(N+3\delta)\over 2}$. 
\end{lemma}
\begin{proof}
Let $a_1=\alpha(0)$ and $a_2=\alpha(k)$. Let $\gamma_1'$ be a geodesic from $a_1$ to a closest point of $gC$. Then $|\gamma_1'|\leq N$. Let $\gamma_1$ be $\gamma_1'$ followed by a vertical geodesic to $H(3\delta)$, a geodesic of length $\leq N+3\delta$. Similarly consider a geodesic $\gamma_2$ from $a_2$ to a vertex of  $H(3\delta)$, so that $|\gamma_2|\leq  N+3\delta$. Let $b_1$ ($b_2$)  be the terminal point of $\gamma_1$ (respectively $\gamma_2$) and $\beta$ be a geodesic between $b_1$ and $b_2$. Since $H^{3\delta}$ is convex, it contains the image of $\beta$. By Lemma \ref{parallel},  the distance between $\alpha (N+3\delta+i)$ and a point of $\beta$ (and hence a point of $H^{3\delta}$) is $\leq 2\delta$ for $0\leq i\leq k-2(N+3\delta)$. Then  $\alpha (N+3\delta+i)\in H^{\delta}$ for $0\leq i\leq k-2(N+3\delta)$. In particular, $\alpha$ restricted to $[N+3\delta, k-(N+3\delta)$ has image in $H^{\delta}$ (as is required in the first part of the lemma). 

Now assume that $\alpha(0)\in H(0)$. Note that $\alpha$ restricted to the interval $[N+3\delta, k-(N+3\delta)]$ has image in $H^\delta$. Let $N_1(\leq N+3\delta)$ be the first integer such that $\alpha$ restricted to $[N_1, k-(N+3\delta)]$ has image in  $H(\delta)$. Let $N_2(\leq N+3\delta)$ be the smallest integer such that $\alpha$ restricted to $[N_1, k-N_2]$ has image in  $H(\delta)$. Let $p$ be the point of $H(0)$ directly below $\alpha(N_1)$. Note that $d(\alpha(0),p)\leq N+4\delta$. Let $L(N)$ be an integer such that for any $h\in G$, two points in $hC$ of distance apart $\leq N+4\delta$ in $X$ are connected by a path in $hC$ of length $\leq L(N)$. Let $\beta$ be a path in $gC$ of length $\leq L(N)$ from $\alpha(0)$ to $p$. (See Figure 4.)

\vspace {.4in}
\vbox to 3in{\vspace {-1.8in} \hspace {-1.95in}
\includegraphics[scale=1]{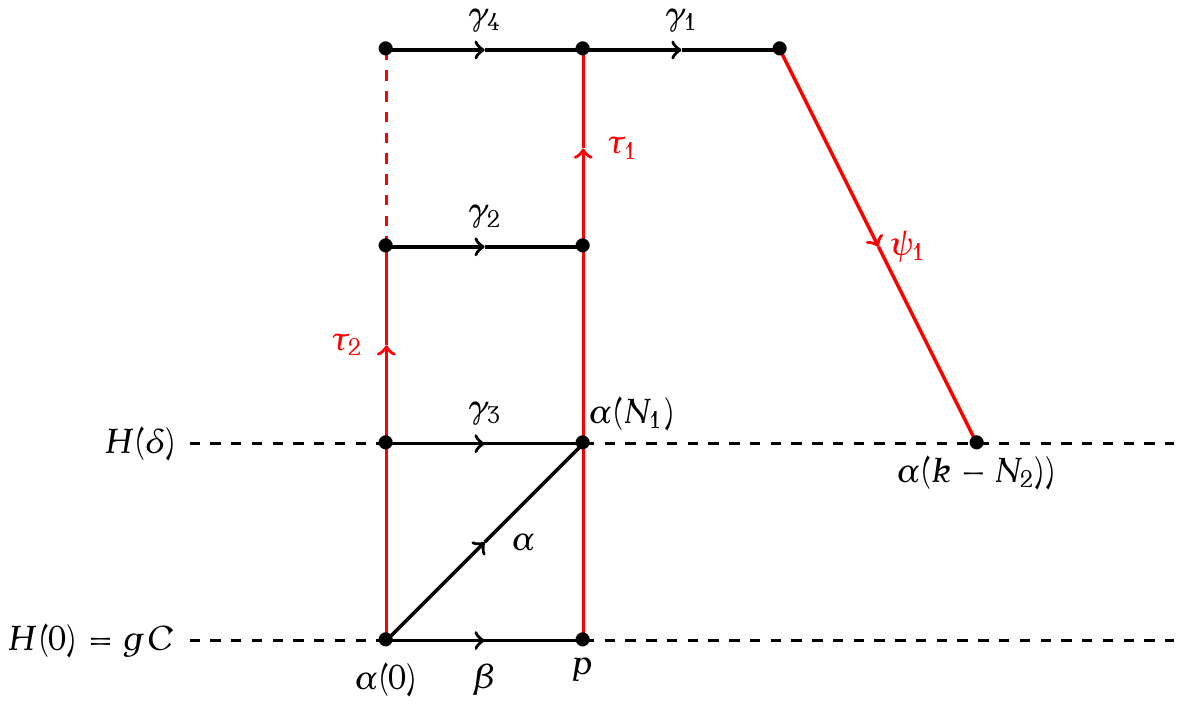}
\vss }
\vspace{-.1in}

\centerline{Figure 4}

\medskip

Let $\alpha'$ be a geodesic obtained from $\alpha$ by replacing $\alpha$ on $[N_1, k-N_2]$ by $(\tau_1,\gamma_1,\psi_1)$ where $\tau_1$ and $\psi_1$ are vertical and $\gamma_1$ is horizontal of length $\leq 3$ (see Lemmas \ref{GM3.10} and \ref{geo}). Let $M(N)$ be the first integer larger than $\delta$ such that $2^{M(N)}\geq L(N)$. Then there is a geodesic $(\tau_2, \gamma_2)$ from $\alpha(0)=\beta(0)$ to a vertex directly above $\alpha(N_1)$ (and hence directly above $p$), where $|\tau_2|= M(N)$ and $|\gamma_2|\leq 1$. If $|\tau_1|\geq  M(N)-\delta$, then the  end point of $\gamma_2$ is on $\tau_1$. That means that the geodesic $\alpha|_{[0,N_1]}$ has at most one horizontal edge. 

{\it Assume for the moment that $|\tau_1|\geq  M(N)-\delta$.} 

If $\alpha|_{[0,N_1]}$ has one horizontal edge. Then there is a path $(\tau_3,\gamma_3)$ from $\alpha(0)$ to $\alpha(N_1)$ where $\tau_3$ is vertical and $\gamma_3$ is an edge. Then $(\tau_3,\gamma_3, \tau_1,\gamma_1, \psi_1)$ is geodesic from $\alpha(0)$ to $\alpha(k-N_2)$. Let $\tau_4$ be the vertical geodesic beginning at $\alpha(0)$ and ending in the same level as the end point of $\tau_1$. Let $\gamma_4$ be the horizontal edge with the same label as $\gamma_3$. Then $(\tau_4,\gamma_4, \gamma_1,\psi_1)$ is geodesic of length $k-N_2$. Then length of the vertical segment $\gamma_4$ is: 
$$|\gamma_4|\geq {k-N_2+\delta-4\over 2}\geq {k-(N+3\delta) +\delta-4\over 2}={k-(N-2\delta+4)\over 2}$$

If $\alpha|_{[0,N_1]}$ is vertical, then the vertical path $\tau_5$ from $\alpha(0)$ to the initial point of $\gamma_1$ is such that $(\tau_5,\gamma_1,\psi_1)$ is geodesic from $\alpha(0)$ to $\alpha(k-N_2)$ and 
$$|\gamma_5|\geq {k-(N+3\delta)+\delta-3\over 2}={k-(N+2\delta+3)\over 2}$$

In either case, there is a geodesic from $\alpha(0)$ to $\alpha(k)$ with initial vertical segment of length ${k-(N+2\delta+4)\over 2}$. We can select: 
$$K_{\ref{Deep}}(N)=N+3\delta(\geq N+2\delta+4)$$ 
{\it We now need to select $k(=|\alpha|)$ large enough to ensure $|\tau_1|\geq M(N)-\delta$.}
Since $k-N_2=N_1+|\tau_1|+|\gamma_1|+|\psi_1|$, $N_i\leq N+3\delta$, $|\gamma_1|\leq 3$ and $|\tau_1|=|\psi_1|$:
$$ |\tau_1|= {k-(N_1+N_2)-|\gamma_1|\over 2}\geq {k-2(N+3\delta)-3\over 2}$$
If $k\geq [2(N+3\delta)+3]+ 2(M(N)-\delta)$ then $|\tau_1|\geq M(N)-\delta$. Finally let 
$$J_{\ref{Deep}}(N)=[2(N+3\delta)+3]+ 2(M(N)-\delta)$$ 
\end{proof}

\section{Linearly Connected Limit Sets of Quasi-isometrically Embedded Subspaces} \label{QIE}
In order to prove our main theorem, we must know that the linearly connected visual metrics on our vertex groups carry over to linearly connected limit sets of their quasi-isometrically embedded images in the cusped space for the over group. This section is devoted to a general result (Theorem \ref{sub}) that implies what we need.

\begin{theorem} \label {sub} 
Suppose $(A,{\bf Q})$ is a relatively hyperbolic pair, $Y$ is a cusped space for $(A,{\bf Q})$ with visual metric $d_1$ on $\partial Y$, and $\partial Y$ is linearly connected with respect to $d_1$. If $A$ is a subgroup of $G$, $(G,{\bf P})$ is relatively hyperbolic with cusped space $X$ and visual metric $d_V$ on $\partial X$, and the map $i:Y\to X$, induced by inclusion is a quasi-isometry onto its image $Y'\subset X$, then the limit set $Z(Y')$ of $Y'$ is linearly connected with respect to $d_V$.  
\end{theorem}
\begin{proof}
First of all, there is a homeomorphism $\hat i:\partial Y\to \partial Y'$ induced by the quasi-isometry $i$  (see Theorem III.H.3.9, \cite{BrHa99})
Let $\ast\in Y$ be the identity vertex. There is a constant $K$ such that if $y_1,y_2\in \partial Y$ then there is a connected set $C(y_1,y_2)$  in $\partial Y$ of diameter $\leq Kd_1(y_1,y_2)$  and containing $y_1$ and $y_2$. 
Let $x_1,x_2\in Z(Y')$, and $y_1,y_2\in \partial Y$ such that $\hat i(y_1)=x_1$ and $\hat i(y_2)=x_2$. 
Recall (Definition \ref{visualm}), there are positive constants $k_1,k_2,k_1',k_2'$ such that for $y_1,y_2\in \partial Y$ and $x_1,x_2\in X$:
$$k_1e^{-(y_1,y_2)_\ast} \leq d_1(y_1,y_2)\leq k_2e^{-(y_1.y_2)_\ast} \hbox{; } k_1'e^{-(x_1,x_2)_\ast} \leq d_V(x_1,x_2)\leq k_2'e^{-(x_1.x_2)_\ast}$$
Then for any $y_3\in C(y_1,y_2)$, 
$$k_1e^{-(y_1.y_3)_\ast} \leq d_1(y_1,y_3)\leq Kd_1(y_1,y_2)\leq k_2Ke^{-(y_1.y_2)_\ast}=e^{ln(k_2K)-(y_1.y_2)_\ast}$$ 
Hence 
$$(A1)\ \ \ \ \ m=:(y_1.y_2)_\ast\leq (y_1.y_3)_\ast+ln({k_2K\over k_1})$$ 
Similarly,
$$(A2)\ \ \ \ \ m=:(y_1.y_2)_\ast\leq (y_2.y_3)_\ast+ln({k_2K\over k_1})$$ 

\begin{lemma} \label{K1} 
There is a constant $K_1$ such that for any $y_3\in C(y_1,y_2)$: 
$$|(y_1.y_2)_\ast -min\{(y_1.y_3)_\ast, (y_2.y_3)_\ast\}|\leq K_1$$
\end{lemma}
\begin{proof} 
First observe:
$$k_1e^{-(y_1.y_2)_\ast}\leq d_1(y_1,y_2)\leq d(y_1, y_3)+d_1(y_2,y_3)\leq k_2(e^{-(y_1,y_3)_\ast}+e^{-(y_2.y_2)_\ast})$$
Then: 
$$e^{-(y_1.y_2)_\ast} \leq {k_2\over k_1}(e^{-(y_1,y_3)_\ast}+e^{-(y_2.y_2)_\ast})\leq {k_2\over k_1}max\{2e^{-(y_1,y_3)_\ast }, 2e^{-(y_2.y_2)_\ast } \}$$
$$=max \{e^{ln({2k_2\over k_1})-(y_1.y_3)_\ast}, e^{ln({2k_2\over k_1})-(y_2.y_3)_\ast }\}$$
$$(y_1.y_2)_\ast\geq min\{(y_1.y_3)_\ast, (y_2,y_3)_\ast\}-ln({2k_1\over k_1})$$
Combining this last inequality with $(A1)$ and $(A2)$:
$$min\{(y_1.y_3)_\ast, (y_2,y_3)_\ast\}-ln({2k_2\over k_1})\leq (y_1.y_2)_\ast\leq min\{(y_1.y_3)_\ast, (y_2,y_3)_\ast\}+ln({k_2K\over k_1})$$
Let $m'=(y_1.y_3)_\ast$ and $m''=(y_2.y_3)_\ast$. This last inequality becomes:
$$ (A3)\ \ \ min\{m',m''\}-ln({2k_2\over k_1}) \leq m\leq min\{m',m''\}+ln({k_2K\over k_1})$$
Simply let $K_1=max\{ln({k_2K\over k_1}), ln({2k_2\over k_1})\}$ to complete the proof of the lemma.
\end{proof}

Let $q_1$ be the quasi-isometry constant for $i$, and $r_1$, $r_2$ and $r_3$ be geodesics at the identity vertex $\ast\in Y$ converging to $y_1$, $y_2$  and $y_3$ respectively. Let $s_1$, $s_2$ and $s_3$ be geodesics at the identity vertex $\ast\in X$ (we use $\ast$ for our base point in both $X$ and $Y$) converging to $x_1$, $x_2$  and $x_3$ respectively. Now there is a constant $q_2$ such that if $r$ is a $q_1$ quasi-geodesic ray at $\ast\in X$ and $s$ is a geodesic at $\ast$ converging to the same boundary point as does $r$, then $r$ and $s$ $q_2$-track one another (Proposition 3.3 \cite{ABC91}). In particular, $i(r_j)$ is $q_2$-tracked by $s_j$, for $j\in \{1,2,3\}$.

\begin{lemma} \label{K2} 
 There is a constant $K_2$ such that if $r_1$ and $r_2$ are geodesic rays at $\ast \in Y$ converging to $y_1$ and $y_2$ respectively, $s_1$ and $s_2$ are geodesic rays at $\ast\in X$ which $q_2$ track $ir_1$ and $ir_2$ respectively, $m=(y_1.y_2)_\ast$, and $d(s_1(m_1),i(r_1( m)))\leq q_2$ for some $m_1\geq 0$, then ($(x_1.x_2)_\ast$ is ``close" to $m_1$):
$$|(x_1.x_2)_\ast -m_1|\leq K_2$$
 By symmetry, if $m_2$ is such that $d(s_2(m_2),i(r_2( m)))\leq q_2$ then
 $$  |(x_1.x_2)_\ast -m_2|\leq K_2 $$
\end{lemma}
\begin{proof} 
Note that  $d(i(a),i(b))\leq d(a,b)$ for all $a,b\in Y$ (since $i$ maps edges to edges). By Remark \ref{sub1}, $d(r_1(m),r_2(m))\leq \delta$ and so $d(ir_1(m),ir_2(m))\leq \delta$.
By the triangle inequality (see Figure 5):
$$(B)\ \ \ \ d(s_1(m_1), s_2(m_2))\leq 2q_2+d(i(r_1(m)), i(r_2(m)))\leq 2q_2+\delta$$
Again by the triangle inequality (with $m_1=d(\ast, s_1(m_1))$ and $m_2=d(\ast, s_2(m_2))$):
 $$m_1-d(s_1(m_1), s_2(m_2))\leq m_2\leq m_1+d(s_1(m_1), s_2(m_2))$$
This last inequality and equation $(B)$ imply:
$$(C)\ \ \ \ \ \ \ \ |m_2-m_1|\leq d(s_1(m_1), s_2(m_2))\leq 2q_2+\delta$$ 
$$d(s_1(m_1), s_2(m_1))\leq d(s_1(m_1),s_2(m_2))+|m_2-m_1|\leq 2d(s_1(m_1),s_2(m_2))$$

\vspace {.1in}
\vbox to 3in{\vspace {-2in} \hspace {-1.8in}
\includegraphics[scale=1]{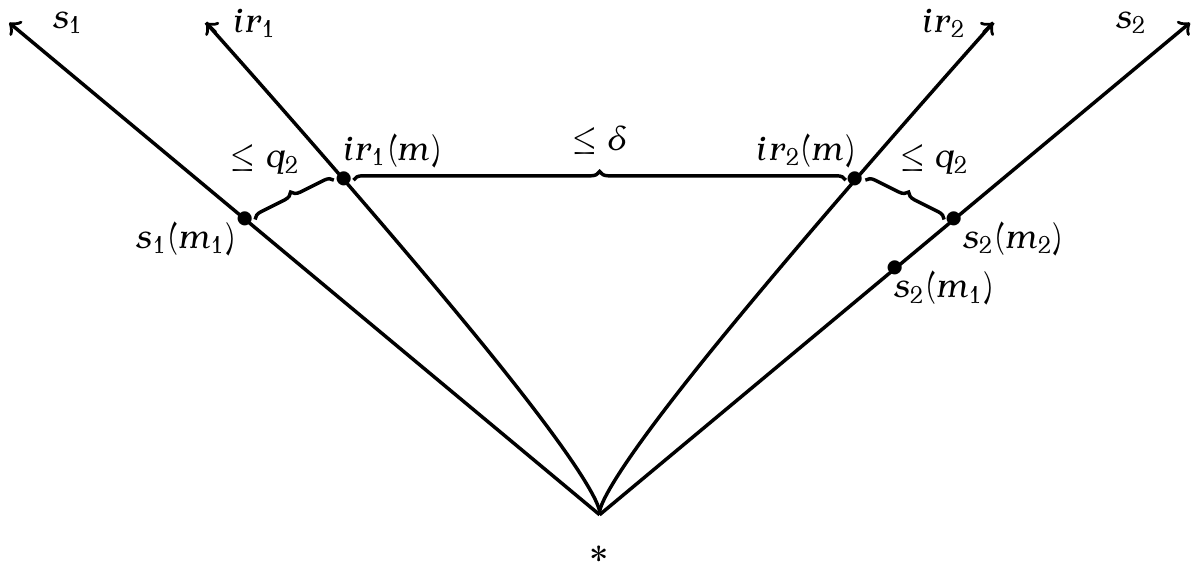}
\vss }
\vspace{-.9in}

\centerline{Figure 5}

\medskip

Combining this last inequality with $(B)$:
$$(D1)\ \ \ \ \ \ d(s_1(m_1), s_2(m_1))\leq 4q_2+2\delta$$ 
Symmetrically:
 $$(D2)\ \ \ \ \ \ d(s_1(m_2), s_2(m_2))\leq 4q_2+2\delta$$ 




\begin{claim}\label{Claim1} 
If  $d(s_1(t), s_2(t))\leq k$, then $(x_1.x_2)_\ast\geq t-{k\over 2}-2\delta$. In particular,  for $t=m_1$ (and $k=(4q_2+2\delta$), $(D1)$ implies $(x_1.x_2)_\ast \geq m_1-2q_2-3\delta$.
\end{claim} 
\begin{proof}
Considering the geodesic triangle with sides $s_1|_{[0,t]}$, $s_2|_{[0,t]}$ and a geodesic connecting $s_1(t)$ and $s_2(t)$. The insize points on $s_1$ and $s_2$ are $s_1(t-{k'\over 2})$ and $s_2(t-{k'\over 2})$ where $k'\leq k$. Then 
$$d(s_1(t-{k'\over 2}), s_2(t-{k'\over 2}))\leq \delta$$ 
Let $l'$ be a geodesic line forming an ideal triangle with $s_1$ and $s_2$. Let $v'$ be the insize point on $l'$ and $\bar m=(x_1.x_2)_\ast$. Note that $d(s_1(\bar m+n), s_2(\bar m+n))\geq 2\delta$ if $n\geq 2\delta$. If $t-{k'\over 2}\geq \bar m+2\delta$ then $d(s_1(t-{k'\over 2}), s_2(t-{k'\over 2}))\geq 2\delta$. But $d(s_1(t-{k'\over 2}), s_2(t-{k'\over 2}))\leq \delta$. Instead, $t-{k\over 2}\leq t-{k'\over 2}\leq \bar m+2\delta$. 
\end{proof}

\begin{claim}\label{Claim2} 
$(x_1.x_2)_\ast \leq m_1+5q_2+4\delta$. 
\end{claim}
\begin{proof}
Let $\bar m=(x_1.x_2)_\ast$ (so that $d(r_1(\bar m),r_2(\bar m))\leq \delta$) and $L=\bar m-m_1$. Our goal is to show: $L\leq  5q_2+4\delta$.  Let  $l$ be a geodesic line forming an ideal triangle with $r_1$ and $r_2$. Then $i(l)$ is a $q_1$ quasi-geodesic. 
Let $v$ be the insize point of $l$ in this triangle so that the points $v$, $r_1(m)$ and $r_2(m)$  are within $\delta$ of one another and so $i(v)$, $i(r_1(m))$ and $i(r_2(m))$ are within $\delta$ of one another as well.  For $j\in \{1,2\}$ let $a_j=ir_j(t_j)$ be a point of $i(r_j)$ such that:
$$d(a_j=ir_j(t_j),s_j(\bar m))\leq q_2$$
We show that $t_1\geq m$ (so that $a_1$ is within $\delta$ of a point $b_1$ of $il$ between $i(v)$ and $x_1$ - see Figure 6). First of all, $\bar m>m_1$ or the Claim is trivial. 
If $t_1<m$, let $\alpha_1$ be a geodesic from $\ast$ to $ir_1(m)$. Since $ir_1$ is a $q_1$ quasi-geodesic, there is a point $p_1$ on $\alpha_1$ such that:
$$d(a_1=ir_1(t_1),p_1)\leq q_2$$ 
Consider the geodesic triangle formed by $\alpha_1$, $\beta_1$ (a geodesic of length $\leq q_2$ from $s_1(m_1)$ to $ir_1(m)$) and $s_1([0,m_1])$. Since $|\beta_1|\leq q_2$, every point of $\beta_1$ is within $\delta+q_2$ of $s_1([0,m_1])$. In particular, $p_1$ is within $\delta +q_2$ of $s_1([0,m_1])$. So $d(s_1(\bar m), ir_1(t_1))\leq q_2$, $d(ir_1(t_1), p_1)\leq q_2$ and $d(p_1, s_1([0,m_1]))\leq q_2+\delta$.  By the triangle inequality, $d(s_1(\bar m), s_1([0,m_1])\leq 3q_2+\delta$, so $d(s_1(\bar m), s_1(m_1))\leq 3q_2+\delta$. Since $s_1$ is geodesic, $L=\bar m-m_1\leq 3q_2+\delta$ and the proof of the Claim is finished. So we may assume:
$$t_1\geq m$$

Next we show that if $a_2=ir_2(t_2)$ then $t_2\geq m$ (so that $a_2$ is within $\delta$ of a point $b_2$ of $il$ between $i(v)$ and $x_2$ -see Figure 6).

First of all we want to see that $\bar m\geq m_2$. We have $\bar m>m_1$. If $m_1\geq m_2$ then certainly $\bar m\geq m_2$, so we may assume that $m_1<m_2$. If $\bar m<m_2$, then $m_1<\bar m<m_2$ and $m_2-m_1\leq 2q_2+\delta$ (equation $(C)$). This implies that $L=\bar m-m_1\leq 2q_2+\delta$ and the Claim is proved. Instead we may assume $\bar m\geq m_2$. 

Replacing $j=1$ with $j=2$, the argument showing $0\leq \bar m-m_1\leq 3q_2+\delta$ shows that $0\leq \bar m-m_2\leq 3q_2+\delta$. Since $|m_2-m_1|\leq 2q_2+\delta$ (equation $(C)$), $L=|\bar m-m_1|\leq 5q_2+2\delta$ and the proof of the Claim is finished. So we may assume:
 $$t_2\geq m$$

Let $b_j$ a point of $i(l)$ within $\delta$ of $a_j$. Then (see Figure 6): 
$$d(s_1(m_1), s_1(\bar m))=L\leq $$
$$d(s_1(m_1), ir_1(m))+d(ir_1(m),i(v))+d(i(v),b_1)+d(b_1,a_1)+d(a_1,s_1(\bar m))$$
So that:
$$L\leq q_2+\delta+d(b_1, i(v))+\delta+q_2$$
Equivalently:
$$(E)\ \ \ \ \ \ d(b_1, i(v))\geq L-2(q_2+\delta)$$

\vspace {.4in}
\vbox to 3in{\vspace {-2in} \hspace {-1.8in}
\includegraphics[scale=1]{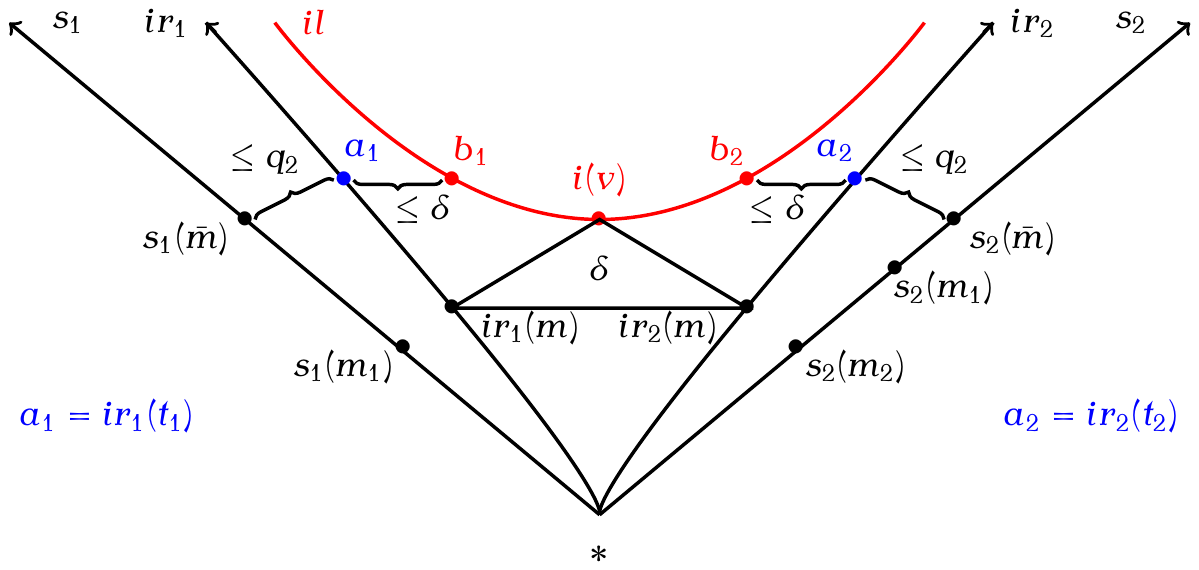}
\vss }
\vspace{-.9in}

\centerline{Figure 6}

\medskip

By the triangle inequality and inequality $(C)$:
$$(F)\ \ \ \ \ d(s_2(m_1),ir_2(m))\leq d(ir_2(m),s_2(m_2))+|m_2-m_1|\leq 3q_2+\delta$$
Next:
$$d(s_2(m_1),s_2(\bar m)=L\leq)$$
$$d(s_2(\bar m),a_2)+d(a_2,b_2)+d(b_2,i(v))+d(i(v), ir_2(m))+d(ir_2(m),s_2(m_1))=$$
$$q_2+\delta+d(b_2,i(v))+\delta+d(ir_2(m), s_2(m_1))$$
Reorganizing terms and applying $(F)$: 
$$(G)\ \ \ \ d(b_2,i(v))\geq L-q_2-2\delta-d(s_2(m_1),i(r_2(m)))\geq L-4q_2-3\delta$$
Let $\tau$ be a geodesic from $b_1$ to $b_2$. Since $il$ is a $q_1$ quasi-geodesic, there is $t\geq 0$ such that $d(\tau(t), i(v))\leq q_2$. By $(E)$ and $(G)$:
$$d(b_1,\tau(t))\geq d(b_1,i(v))-q_2\geq L-3q_2-2\delta \hbox{ and }$$
$$ d(\tau(t), b_2)\geq d(b_2,i(v))-q_2\geq L-5q_2-3\delta$$
Combining:
$$(H)\ \ \ \ d(b_1,b_2)=d(b_1,\tau(t))+d(\tau(t),b_2)\geq 2L-8q_2-5\delta$$

But, since $d(s_1(\bar m), s_2(\bar m))\leq \delta$,
$$ d(b_1,b_2)\leq d(b_1,s_1(\bar m)) +d(s_1(\bar m),s_2(\bar m))+d(s_2(\bar m),b_2)\leq 2q_2+3\delta$$
Combining this last inequality with $(H)$:
$$2L-8q_2-5\delta\leq d(b_1,b_2)\leq 2q_2+3\delta$$
$$L\leq 5q_2+4\delta$$
This completes the proof of the claim
\end{proof}
To complete the proof of the Lemma, simply let 
$$K_2=5q_2+4\delta$$ 
to obtain the first inequality. By symmetry the second inequality is true as well. 
\end{proof}

\begin{lemma}\label{K3} 
There is a constant $K_3$ such that for any $x_3\in \hat i(C(y_1,y_2))$:
$$min \{|(x_1.x_2)_\ast-(x_1.x_3)_\ast|, |(x_1,x_2)_\ast -(x_2.x_3)_\ast|\}\leq K_3$$
\end{lemma} 
\begin{proof}
Let  $r_3$ be a geodesic at $\ast\in Y$ converging to $y_3\in C(y_1,y_2)$ where $\hat i(y_3)=x_3$. Let $s_3$ be a geodesic at $\ast\in X$ that $q_2$ tracks $i(r_3)$.  Then $s_3$ converges to $x_3=\hat i(y_3)$. Let $m'=(y_1.y_3)_\ast$ and let $m_1'>0$ be such that $d(s_1(m_1'), ir_3(m'))\leq q_2$. Let $m''=(y_2.y_3)_\ast$ and let $m_1''>0 $ be such that $d(s_2(m_1''), ir_3(m'')\leq q_2$. See Figure 7.

By Lemma \ref{K1}:
$$|(y_1.y_2)_\ast -min\{(y_1.y_3)_\ast, (y_2.y_3)_\ast\}|\leq K_1$$
Say $|(y_1.y_2)_\ast -(y_1.y_3)_\ast|\leq K_1$. Since $s_1$ is geodesic, the triangle inequality implies: 
$$|(x_1.x_2)_\ast-(x_1.x_3)_\ast|=d(s_1((x_1.x_2)_\ast), s_1((x_1.x_3)_\ast))\leq d( s_1((x_1.x_2)_\ast), s_1(m_1))+$$
$$d(s_1(m_1), ir_1(m))+d(ir_1(m), ir_1(m'))+d(ir_1(m'), s_1(m_1'))+d(s_1(m_1'),s_1((x_1.x_3)_\ast))$$

\vspace {.4in}
\vbox to 3in{\vspace {-1.8in} \hspace {-2.1in}
\includegraphics[scale=1]{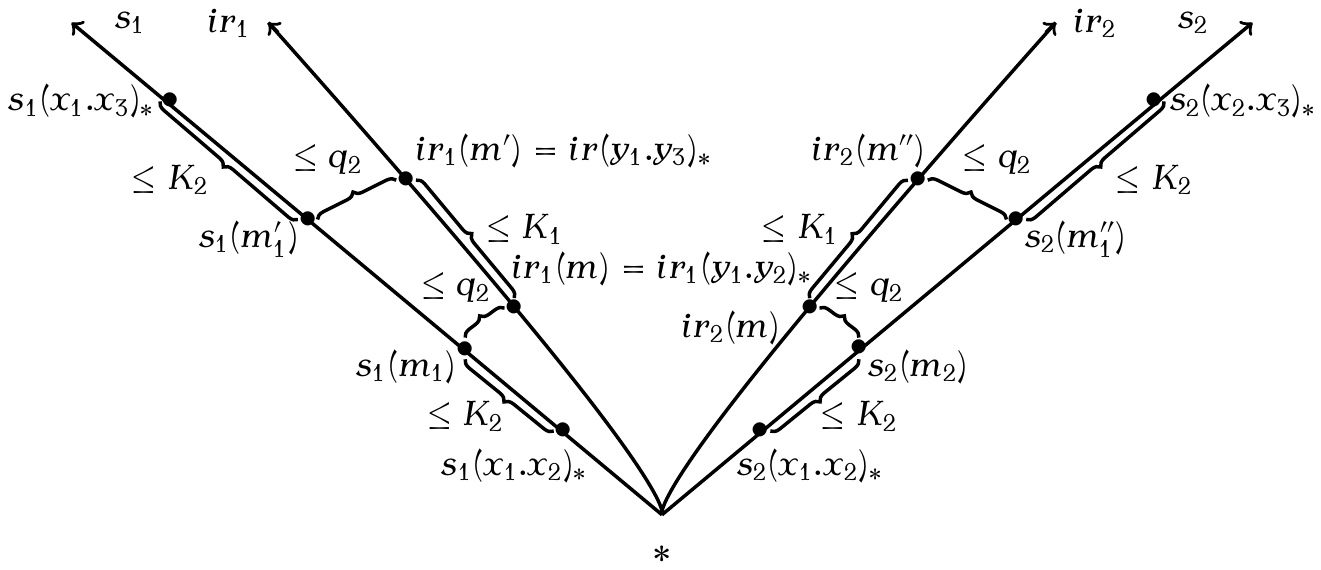}
\vss }
\vspace{-.8in}

\centerline{Figure 7}

\medskip

Simplifying via Figure 7:
$$|(x_1.x_2)_\ast-(x_1.x_3)_\ast|\leq |(x_1.x_2)_\ast-m_1|+q_2+K_1+q_2+|(x_1.x_3)_\ast-m_1'|$$
By Lemma \ref{K2}: $ |(x_1.x_2)_\ast-m_1|\leq K_2$, $|(x_1.x_2)_\ast-m_2|\leq K_2$, $ |(x_1.x_3)_\ast-m_1'|\leq K_2$ and $|(x_2.x_3)_\ast-m_1''|\leq K_2$ so:
$$|(x_1.x_2)_\ast-(x_1.x_3)_\ast|\leq 2K_2+2q_2+K_1$$
Similarly  if $|(y_1.y_2)_\ast -(y_2.y_3)_\ast|\leq K_1$, then
$$|(x_1.x_2)_\ast-(x_2.x_3)_\ast|\leq 2K_2+2q_2+K_1$$
Select $K_3=2K_2+2q_2+K_1$ to finish the lemma.
\end{proof}

Now we finish the proof of the Theorem. Suppose $x_1,x_2\in Z(Y')$. Let $y_1, y_2\in \partial Y$ be such that $\hat i(y_1)=x_1$ and $\hat i(y_2)=x_2$. Consider $x_3$ in the connected set $\hat i(C(y_1,y_2))$ (which contains $x_1$ and $x_2$). Then $x_3=\hat i(y_3)$ for some $y_3\in C(y_1,y_2)$. 

By Lemma \ref{K3}, either $(x_1.x_3)_\ast-K_3\leq (x_1.x_2)_\ast\leq (x_1.x_3)_\ast+K_3$ or 
$(x_2.x_3)_\ast-K_3\leq (x_1.x_2)_\ast\leq (x_2.x_3)_\ast+K_3$. Assume the former. Then:
$$e^{K_3}e^{-(x_1.x_3)_\ast}\geq e^{-(x_1.x_2)_\ast}\geq e^{-K_3}e^{-(x_1.x_3)_\ast}$$
Since $d_V$ is a visual metric on $\partial X$, if $x,y\in \partial X$:
$$k_1'e^{-(x.y)_\ast}\leq d_V(x,y)\leq k_2'e^{-(x.y)_\ast}$$
These last two inequalities imply: 
$$d_V(x_1, x_3)\leq k_2'e^{-(x_1.x_3)_\ast}\leq k_2'e^{K_3}e^{-(x_1.x_2)_\ast}\leq {k_2'\over k_1'}e^{K_3}d_V(x_1,x_2)$$ 
Similarly, if 
$(x_2.x_3)_\ast-K_3\leq (x_1.x_2)_\ast\leq (x_2.x_3)_\ast+K_3$ then 
$$d(x_2, x_3)\leq {k_2'\over k_1'}e^{K_3}d(x_1,x_2)$$ 
In any case, the diameter of the connected set $\hat i(C(y_1,y_2))$ (containing $x_1$ and $x_2$)  is $\leq d(x_1,x_2)({k_2'\over k_1'}e^{K_3} +1)$, and $Z(Y')$ is linearly connected.
\end{proof}

\section{Piecewise Visual Linearly Connected Metrics} \label{PCM}

The proof of the main theorem is nearly identical to that of the simplest case $G=A\ast_CB$, where $G$, $A$ and $B$ are all hyperbolic relative to $C$, but the notation in this basic case is substantially easier to assemble. We prove the base case and then comment on the minor adjustments required to prove the theorem in the case when $G$ is hyperbolic relative to $B$, and  $A$ is hyperbolic relative to $C$; the case when $G$ is an HNN extension $A\ast_C$ and both $G$ and $A$ are hyperbolic relative to $C$; and the finally general graph of groups situation. In all of our proofs we use $e^{-(r.s)_p}$ instead of $a^{-(r.s)_p}$ for a general real number $a>1$ although our proofs work equally well with any fixed base $a>1$.

We are in the situation where $G=A\ast_CB$, the groups $G$, $A$ and $B$ are hyperbolic relative to  $C$ and the space $\partial (G,C)$ is connected. The spaces  $\partial (A,C)$  and $\partial (B,C)$ are connected, locally connected do not have cut points.
Assume that $X$ is the cusped space for $(G,C)$ derived from a finite presentation that has generators $S_A$, $S_B$ and $S_C$ for $A$, $B$ and $C$ respectively. Since $\partial X$ is connected, $X$ is 1-ended. 
The boundary of $X$ is a tree of spaces with additional {\it ideal} points. The tree $\mathcal T$ is the Bass-Serre tree for $A\ast_CB$ (or the graph of group decomposition of $G$ in general).  
Let $vH$ be a vertex group of $\mathcal T$ (so $v\in G$ and $H$ is either $A$ or $B$). Let $Z(vH)$ be the limit set of $vH\subset X$. 
Then $\partial X$ is the union of the $Z(vH)$ along with the ideal points. (In the case $G=A\ast_VB$ is hyperbolic relative to $B$, each $Z(vB)$ is a single point. When $G=A\ast_C$, $X$ is the union of the $Z(vA)$ along with ideal points.) The distinct sets $Z(vA)$ and $Z(wB)$ intersect non-trivially if and only if  $vA\cap wB\ne \emptyset$ if and only if there is $u\in G$ such that $vA\cap wB =uC$. (If $G=A\ast_C$, then the distinct sets $Z(vA)$ and $Z(wA)$ intersect non-trivially if and only if there is $u\in G$ such that $vA=uA$ and $wA=utA$ where $t$ is the stable letter of $A\ast_C$.) 
Each $uC$ has limit set equal to a cut point in $\partial X$ which belongs to and separates the sets $Z(vA)$ and $Z(wB)$. (In the case $G=A\ast_C$, $A$ contains (isomorphic) associated subgroups $C_1$ and $C_2$ and the stable letter $t$ of the HNN extension conjugates $C_1$ to $C_2$. Then for any $u\in G$, $Z(uC_1)=Z(utC_1)$ is a cut point in $\partial X$ separating $Z(uA)$ and $Z(utA)$.) Since $Z(vA)$ and $Z(vB)$ are homeomorphic to $\partial (A, C)$ and $\partial (B,C)$ respectively, $Z(vA)$ and $Z(wB)$ contain no cut points. We use $d$ for the metric on $X$ and $d_V$ for the visual metric on $\partial X$. 

The proof in one direction of the next result uses the fact that $\partial (A,C)$ and $\partial (B,C)$ do not have cut points. 
\begin{lemma}\label{cut1}
Suppose $x\ne y\in\partial X$, and $l$ is a geodesic line in $X$ from $x$ to $y$. 
Then the coset $vC$ of $X$ separates the ends of the line $l$ if and only if the limit set of $vC$ (a single point) separates $x$ and $y$ in $\partial X$. 
\end{lemma}
\begin{proof}
Suppose $vC$ separates the ends of the line $l$ (so that there is an integer $k$ such that $l([k,\infty))$ and $l((-\infty,-k])$ are in different components of $X-vC$). Suppose there is a path $\alpha$ in $\partial X$ from $x$ to $y$ avoiding $c$, the limit set of $vC$. Let $a_t$ be a geodesic ray based at $p$, so that  $a_t\in \alpha(t)$ (so $a_0\in x$ and $a_1\in y$. There must be an integer $m$ such that for $t\in [0,1]$ and all $j\geq m$, $a_t(j)$ is not in the $\delta$-neighborhood of $vC$ (otherwise $c$ is in the limit set of the union of the images of the $a_t$, which is the image of $\alpha$). Choose $k_1$ and $k_1'$ such that $d(a_0(k_1), l(-k_1'))\leq \delta$,  $d(a_1(k_1), l(k_1'))\leq \delta$, $k_1'\geq k$ and $k_1>m$.
Choose a sequence of points $r_0=a_0,\ldots , r_n=a_1$ so that $d(r_i(k_1),r_{i+1}(k_1))\leq \delta$. Let $\alpha_i$ be a path of length $\leq \delta$ from $r_i(k_1)$ to $r_{i+1}(k_1)$. Let $\beta_0$ be a path of length $\leq \delta$ from $l(-k_1')$ to $r(k_1)$ and $\beta_1$ be a path of length $\leq \delta$ from $l(k_1')$ to $s(k_1)$. The path $(\beta_0, \alpha_0,\ldots, \alpha_{n-1}, \beta_1^{-1})$ from $l(-k_1')$ to $l(k_1')$ avoids $vC$, contrary to our assumption. This proves the first half of our lemma.

Next suppose $S=\{\ldots, v_{-1}C,v_0C,v_1C,\ldots\}$ is the set of all cosets (subsets of $X$) that separate the ends of $l$, ordered according the Bass-Serre tree structure of $A\ast_CB$ (this set might be finite, infinite or bi-infinite, depending on whether $x$ and $y$ are ideal points or belong to $gA$ or $gB$ for some $g\in G$). We consider the case $S=\{v_0C, v_1C,\ldots\}$ (all other cases can be resolved by the techniques used in this case). See Figure 8.

\vspace {.4in}
\vbox to 3in{\vspace {-2in} \hspace {-1.8in}
\includegraphics[scale=1]{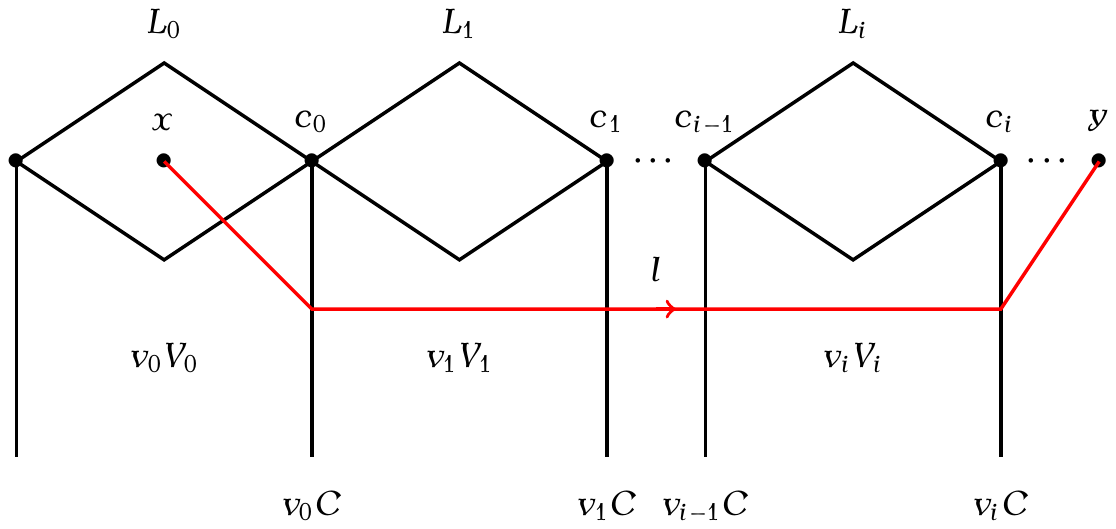}
\vss }
\vspace{-.9in}

\centerline{Figure 8}

\medskip

Note that (for $V_i$ alternating between $A$ and $B$) $v_0C$ and an initial segment of $l$ belong to $v_{0}V_0$, 
 and for $i> 0$, $v_{i-1}C\cup v_{i}C\subset v_{i}V_i$.  Let $L_i$ be the limit set of $v_iV_i$ (so $L_i$ is homeomorphic to either $\partial (A,C)$ or $\partial(B,C)$), and $c_i$ be the cut point of $\partial X$ corresponding to $v_iC$. Then $x,c_0\in L_0$,  $c_{i-1},c_{i}\in L_{i}$, and $L_i\cap L_{i+1}=c_i$. Recall that no point of $L_i$ separates $L_i$, so $L_i-\{c_{i-1}\}$ and $L_i-\{c_{i}\}$ are connected. Then $c_i$ separates the connected sets $L_0\cup \cdots L_{i-1} \cup (L_{i}-\{c_i\})$ and $\{y\}\cup (L_{i+1}-\{c_i\})\cup L_{i+2}\cup \cdots$. Now $L=\{y\}\cup L_0\cup L_1\cdots$ is a connected set containing $x$ and $y$. Suppose $c$ separates $x$ and $y$ in $\partial X$. We must show that $c=c_i$ for some $i$. Suppose $c\ne c_i$ for all $i$. Certainly $c\in L$ and so $c\in L_i$ for some $i$. If $c\in L_0$, then $c$ does not separate $L_0$ (no $L_i$ has a cut point). But then $x$ and $y$ belong to $ \{y\}\cup (L_0-\{c\})\cup C_1\cup C_2\cup\cdots$ a connected set, contrary to our assumption that $c$ separates $x$ and $y$. If $c\in L_i$ for $i>0$, $c$ does not separate $L_i$ so $x$ and $y$ belong to the connected set $\{y\}\cup L_0\cup \cdots L_{i-1} \cup (L_i-\{c\})\cup L_{i+1}\cup \cdots$ contrary to our assumption that $c$ separates $x$ and $y$. (In the general graph of groups case, the only difference is that the members of the set $S$ are cosets of various edge groups.) 
\end{proof}
Suppose $x,y\in \partial X$. The Bass-Serre tree $\mathcal T$ for the decomposition $A\ast_CB$ gives a unique (possibly bi-infinite) ordering of the set of cut points $C_{x,y}=\{\ldots , c_{-1}, c_0, c_1,\ldots \}$ of $\partial X$ that separates $x$ from $y$, where the sets $\{x\}\cup \{\ldots, c_{i-2},c_{i-1}\}$ and $\{c_{i+1},c_{i+1},\ldots\}\cup \{y\}$ belong to distinct components of $\partial X-\{c_i\}$ for all $i$. Observe that $\{c_i,c_{i+1}\}$ is a subset of the limit set of $v_iH_i$ for some $v_i\in G$ and $H_i\in\{A,B\}$. Also, $H_i\ne H_{i+1}$ (so the $H_i$ alternate between $A$ and $B$).

\begin{lemma}\label{Tree}
Suppose $r$ is a geodesic ray in $X$. Then $r$ determines a geodesic ray $\hat r$ in $\mathcal T$ such that $r$ crosses  $gC$ (begins on one side of $gC$ and eventually ends up on another side)  if and only if $\hat r$ contains the edge $gC$ of $\mathcal T$ (unless $gC$ is the first edge of $\hat r$). 
\end{lemma}

In order to define our metric on $\partial X$ we must consider three cases and show the corresponding series converge. 

\begin{definition} \label{pvmetric} 
Let $d$ denote the metric on $X$. Let $d_V\equiv d_\ast$ be a visual (inner product) metric on $\partial X$, based at $\ast$ (the identity vertex of $X$); so there are constants $k_1$ and $k_2$ such that if $x,y\in \partial X$, then $k_1 e^{-(x. y)_\ast}\leq d_V(x,y)\leq k_2e^{-(x. y)_\ast}$. Another (potential) metric is now defined on $\partial X$. We need to consider 3 cases. 

(1) If neither $x$ nor $y$ is an ideal point, then $C_{x,y}=\{c_0,c_1,\ldots, c_n\}$ is finite. Define $d_L(x,y)=d_V(x,c_0) +d_V(c_0,c_1)+\cdots +d_V(c_{n-1},c_n)+d_V(c_n,y)$. In particular, if $x$ and $y$ belong to the limit set of $vH$ for $v\in G$ and $H\in \{A,B\}$, then $d_L(x,y)=d_V(x,y)$.

(2) If $x$ is ideal and $y$ is not, then $C_{x,y}=\{\ldots, c_{-1}, c_0\}$ and we define $d_L(x,y)=(\sum_{i=0}^{-\infty} d_V(c_i,c_{i-1})+d_V(c_0,y))$. Similarly if $y$ is ideal and $x$ is not. 

(3) If both $x$ and $y$ are ideal, then $C_{x,y}=\{\ldots, c_{-1}, c_0,c_1,\ldots \}$ and we define $d_L(x,y)=\sum_{i=-\infty}^{\infty} d_V(c_i,c_{i-1})$. 
\end{definition}

Note that if $x,y,z\in \partial X$, then $C_{x,z}$ is an initial segment of $C_{x,y}$ followed by a terminal segment of $C_{y,z}$, so that if all series in the above definition converge, then $d_L$ is indeed a metric (see Lemma \ref{conv}).

\begin{lemma}\label{separate} 
Suppose $a_1=[r_1]$, and  $a_2=[r_2]$ are distinct points of   $\partial X$ (based at $p\in X$) and $l$ is a geodesic line from $a_1$ to $a_2$. Let $z_1$ on $r_1$, $z_2$ on $r_2$ and $z_3$ on $l$ be internal points of $\triangle (r_1,r_2,l)$ (see Lemma \ref{Yes}).   If $b$ is a vertex of $l$ between $z_3$ and $a_1$  and $d(z_3,b)=K$, then for any point $y$ of $r_2$, $d(y,b)\geq K-2\delta$.
\end{lemma}
\begin{proof}
If $y$ is a point of $(z_2,a_2)$ and $d(y,b)<K-2\delta$, let $y'$ be the corresponding point of $(z_3,a_2)$ so that $d(y,y')\leq \delta$. Then $d(y',b)<K-\delta$ which is nonsense. See Figure 9.

\vspace {.4in}
\vbox to 3in{\vspace {-2in} \hspace {-1.2in}
\includegraphics[scale=1]{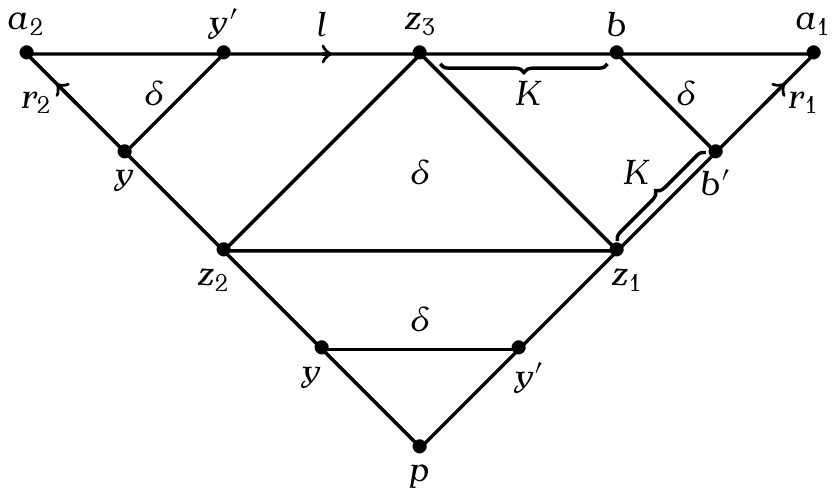}
\vss }
\vspace{-1.2in}

\centerline{Figure 9}

\medskip

If $y$ is a point of $(z_2, p)$ and $d(y,b)<K-2\delta$, let $y'$ be the corresponding point to $y$ on $(z_1,p)$ and $b'$ be the corresponding point to $b$ on $(z_1,a_1)$. Then  $d(y,y')\leq \delta$, $d(b,b')\leq \delta$ and $d(b',z_1)=K$. Then $K\leq d(b', y') \leq d(b',b)+d(b,y)+d(y,y')< K$. 
\end{proof}

Recall that $\mathcal T$ is the Bass-Serre tree for $A\ast_CB$ with vertices labeled $gA$ and $gB$ for $g\in G$ and edges labeled $gC$. For $D\in \{A, B,C\}$ and $g\in G$, the stabilizer of $gD$ is $gDg^{-1}$. 
Say $gA\ (gB)$ is a vertex of $\mathcal T$, then there is a unique edge of $\mathcal T$ containing the vertex $gA\ (gB)$ and separating it from a vertex of $C$.  (If  $gA\ne A$ this edge separates $gA$ from both vertices of $C$).

\begin{lemma}\label{approx} 
Let $g$ be an element of $G$ and $q\in gC$ be a closest point of $gC$ to $\ast$. Suppose the edge $qC$ of $\mathcal T$ (with vertices $qA$ and $qB$) separates $qA$ from a vertex of the edge $C$. Then for any distinct points $a_1,a_2$ in the limit set $Z (qA)\subset \partial X$:
$$ d(\ast, q)+(a_1.a_2)_q -(26\delta+12)\leq (a_1. a_2)_\ast\leq d(\ast, q)+(a_1 a_2)_q +(26\delta+12)$$
Equivalently:
$$e^{-d(\ast,q)}e^{-(a_1. a_2)_q} e^{26\delta+12}\geq e^{-(a_1. a_2)_\ast}
\geq  e^{-d(\ast,q)}e^{-(a_1. a_2)_q} e^{-(26\delta+12)}$$
\end{lemma}
\begin{proof}
For $i\in \{1,2\}$ let $r_i$ be a geodesic ray at $\ast\in X$ converging to $a_i$ and let $s_i$ be a geodesic ray at $q\in X$ converging to $a_i$. Let $p_i$ be the first point of $r_i$ in $gC$. By Lemma \ref{close}, $d(p_i,q)\leq 6\delta+4$ for $i\in \{1,2\}$. This implies:

(1) Each point of $s_1$ is within $7\delta+4$ of a point of the subsegment $(p_1, a_1)$ of $r_1$, and each point of $(p_1,a_1)$ is within $7\delta+4$ of a point of $s_1$. Similarly for $s_2$ and $(p_2,a_2)$. 

If $l$ is a  geodesic line in $X$ from $a_1$ to $a_2$, then Lemma \ref{Yes} gives a vertex $w_3\in l$, such that $w_1$ on $r_1$, $w_2$ on $r_2$ and $w_3$ on $ l$  are internal points of the ideal geodesic triangle $\triangle (s_1,s_2,l)$. See Figure 10. 

\vspace {.4in}
\vbox to 3in{\vspace {-2in} \hspace {-1.6in}
\includegraphics[scale=1]{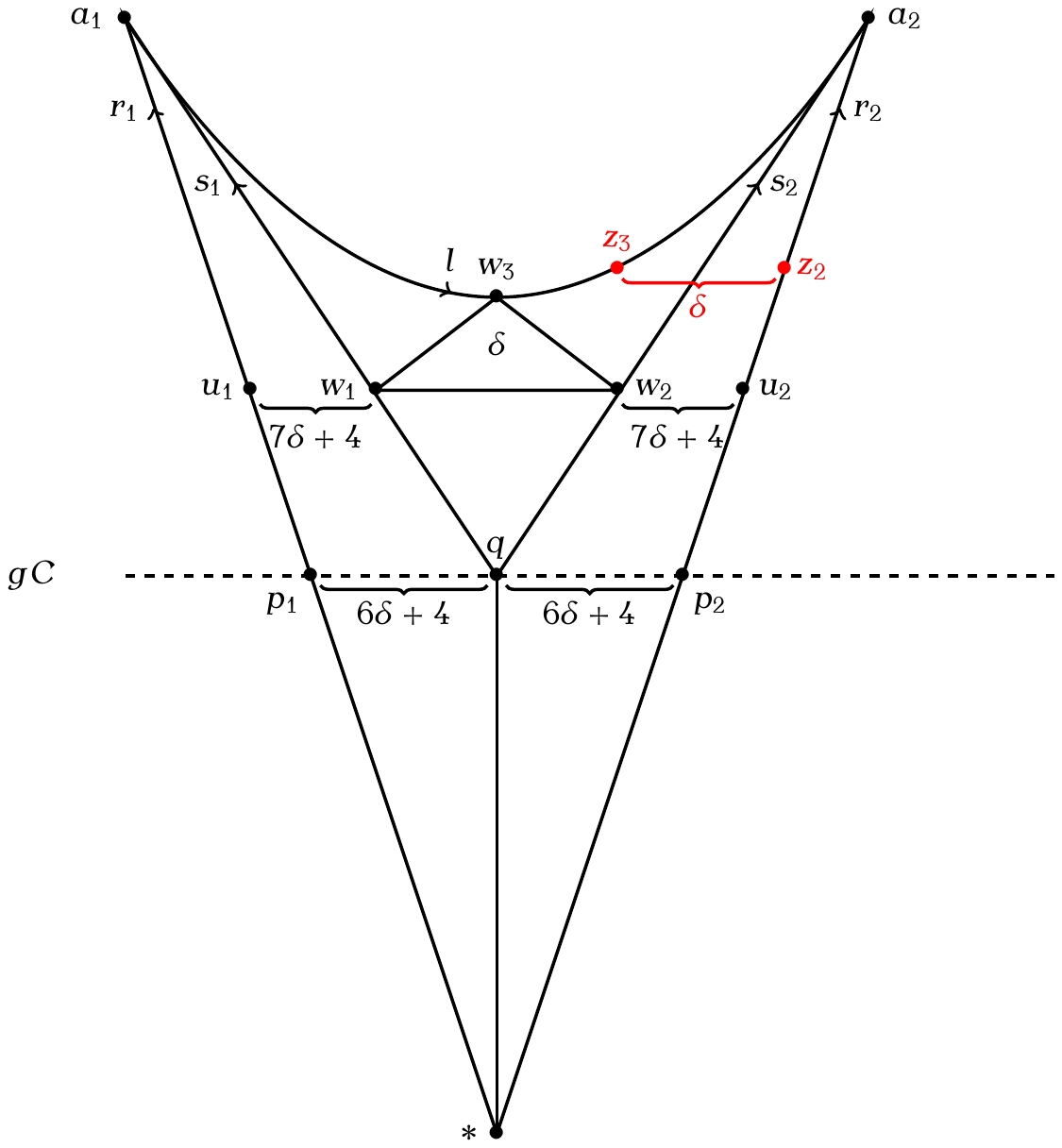}
\vss }
\vspace{1.5in}

\centerline{Figure 10}

\medskip

Then for $i,j\in\{1,2,3\}$, $d(w_i,w_j)\leq \delta$, (see Remark \ref{sub1}) and: 
$$(a_1. a_2)_q=d(q,w_1)=d(q,w_2)$$
By (1), if $i\in \{1,2\}$, there is a point $u_i$ on $r_i$ such that $d(u_i,w_i)\leq 7\delta+4$. 
Let $\{z_1,z_2,z_3\}$ be internal points of the ideal geodesic triangle with sides $r_1$, $r_2$ and $l$, with  $z_1$ on $r_1$, $z_2$ on $r_2$  and $z_3$ on $l$ such that: 
$$(a_1. a_2)_\ast=d(\ast,z_1)=d(\ast,z_2)$$

Without loss, assume that $z_3$ is between $w_3$ and $a_2$.  
Apply Lemma \ref{separate}, to $\triangle (r_1,r_2, l)$ with $w_3$ playing the roll of $b$. Then $d(w_3, r_2)$ (the distance from $w_3$ to the ray $r_2$) is $\geq d(w_3,z_3)-2\delta$.  
So $d(w_3,z_3)-2\delta \leq d(w_3,r_2)\leq d(w_3,w_2)+d(w_2,u_2)\leq 8\delta+4$. Then: 
$$d(w_3,z_3)\leq 10\delta+4$$

Now: 
$$d(z_2, u_2)\leq d(z_2,z_3)+d(z_3,w_3) +d(w_3,w_2)+d(w_2,u_2)\leq 19\delta +8$$

Combining we have:
$$d(\ast,q)+d(q, w_2)+d(w_2,u_2)+d(u_2,z_2)\geq d(\ast,z_2)=(a_1. a_2)_\ast$$ 
$$d(\ast,q)+(a_1. a_2)_q+(7\delta+4)+(19\delta+8)\geq (a_1. a_2)_\ast$$ 
$$ (a_1. a_2)_\ast\leq d(\ast,q)+(a_1. a_2)_q+(26\delta+12)$$
This completes one side of our inequality.
$$d(\ast,q)+(a_1. a_2)_q=d(\ast,q)+d(q, w_2)\leq $$
$$(d(\ast, p_2)+ d(p_2,q))+(d(q,p_2) +d(p_2,z_2)+d(z_2,z_3)+d(z_3,w_3)+d(w_3,w_2))$$
$$=[d(\ast, p_2)+d(p_2,z_2)]+2(6\delta+4)+\delta+(10\delta+4)+\delta=(a_1. a_2)_\ast +24\delta+8$$
Finally:
$$d(\ast,q)+(a_1. a_2)_q-(24\delta+8)\leq (a_1. a_2)_\ast$$
\end{proof}

If $S\subset \partial X$ let $D_S$ be the diameter of $S$ with respect to $d_V$. If $S\subset X$, let $D_S$ be the diameter of the limit set of $S$ in $\partial X$. The constants $k_1$ and $k_2$ are defined in Definition \ref{pvmetric}. Let $X_A$ be the cusped space for $A$ inside of $X$ and $Z(W)$ the limit set of $W\subset X$. 

\begin{lemma}\label{compare} 
Suppose (the edge) $qC$ separates (the vertex) $qA$ from a vertex of (the edge) $C$ in $\mathcal T$ the Bass-Serre tree for $A\ast_CB$  and $q$ is a closest point of $qC$ to $\ast$ (in $X$). Then for $S$ a subset of $Z(X_A)$  (or $S\subset X_A$):
$$D_{qS}\leq {k_2\over k_1}e^{(26\delta+12)-d(\ast, q)}D_S$$ 
(Similarly for $B$.) 
\end{lemma}
\begin{proof}
Let $a_1,a_2\in qS\subset \partial X$. By Lemma \ref{approx}: 
$$d_V(a_1,a_2)\leq k_2e^{-(a_1. a_2)_\ast}\leq k_2e^{-d(\ast, q)-(a_1. a_2)_q +(26\delta+12)}$$
The set  $S$ contains $q^{-1}a_1$ and $q^{-1}a_2$.  
$$e^{-(a_1.a_2)_q}=e^{(q^{-1}a_1.q^{-1}a_2)_\ast}\leq {1\over k_1} d_V(q^{-1}a_1, q^{-1}a_2)\leq {1\over k_1}D_S$$  
Combining inequalities,   
$$d_V(a_1,a_2)\leq {k_2\over k_1} e^{(26\delta+12)-d(\ast,q)} D_S$$ 
Similarly if $S$ is a subset of $X_A$, $Z(X_B)$ or $X_B$.
\end{proof}

\begin{theorem} \label{conv} 
Each series involved in the definition of $d_L$ on $\partial X$ is convergent and so $d_L$ is a metric. 
\end{theorem}
\begin{proof}
It suffices to consider case 2. Say $C_{x,y}=\{c_0,c_1,\ldots\}$. Choose $q_i\in G$ such that the limit set of $q_iC$ is $c_i$ and $q_i$ is a closest point of $q_iC$ to $\ast$. Since $r$ eventually $\delta$-fellow travels with an end of $l$, there is $n\geq 0$ such that 
for $i\geq n$, $q_iC$ separates $\ast$ and $q_nC,\ldots, q_{i-1}C$ from $q_{i+1}C, q_{i+2}C, \ldots$ and: 
$$(\dagger) \ \ \ \ \ \ \ \ \ \ \ \ d(q_i,\ast)<d(q_{i+1},\ast)$$ 
It is enough to show $\sum_{i=n}^\infty d_V(c_i,c_{i+1})$ converges. 
For $i\geq n$, let $q_iE_i$ for some $E_i\in\{A,B\}$ be the coset containing $q_iC$ and $q_{i+1}(C)$. Let $D$ be the maximum of $\{D_A,D_B\}$. By Lemma \ref{compare}: 
$$d_V(c_i,c_{i+1})\leq {k_2\over k_1}e^{(26\delta+12)-d(\ast, q_i)}D$$
$$\sum_{i=n}^\infty d_V(c_i,c_{i+1}) \leq {k_2\over k_1}e^{(26\delta+12)}D\sum_{i=n}^\infty e^{-d(\ast, q_i)} $$
By $(\dagger)$ this last series is convergent. 
\end{proof}

Before leaving this section, we need one more result that will imply $d_L$ is a linearly connected metric on $\partial X$, (once we establish that $d_L$ generates the same topology on $\partial X$ as does $d_V$). By Theorem \ref{sub} the limit set of a cusped space for $gA$ or $gB$ (in $X$) is linearly connected with respect to $d_V$ or $d_L$ for any $g\in G$. Let $q_A$ be the linear connectivity constant for the limit set $Z(X_A)$ of the cusped space $X_A$ for $A$ (in $X$). Recall, if $S\subset X$, then $D_S$ is the diameter of the limit set of $S$ in $\partial X$ with respect to $d_V$. 

\begin{lemma}\label{LinCon} 
Define $K_{\ref{LinCon}}={k_2\over k_1}e^{26\delta+12}$. Let $x_1,x_2\in Z(X_A)$ (the limit set for the cusped space $X_A\subset X$ for $A$) and let $C(x_1,x_2)$ be a connected subset of $Z(X_A)$ containing $x_1$ and $x_2$ such that $D_{C(x_1,x_2)}\leq q_Ad_L(x_1,x_2)$. If $g$ is a closest point of $gA$ to $\ast$, then $D_{gC(x_1,x_2)}\leq (K_{\ref{LinCon}})^2q_Ad(gx_1,gx_2)$. Similarly for $B$ and $X_B$. 
\end{lemma}
\begin{proof}
The metrics $d_L$ and $d_V$ agree on all subspaces in the proof of this lemma.
By Lemma \ref{compare}:
$$D_{gC(x_1,x_2)}\leq Ke^{-d(\ast,g)}D_{C(x_1,x_2)}\leq Ke^{-d(\ast,g)}q_Ad_L(x_1,x_2)$$
Since $d_L$ and $d_V$ agree on the limit set $Z(X_A)$, $d_L(x_1,x_2)\leq k_2e^{-(x_1.x_2)_\ast}$ and:
$$D_{gC(x_1,x_2)}\leq Ke^{-d(\ast, g)}q_A k_2e^{-(x_1.x_2)_\ast}=Ke^{-d(\ast, g)}q_A k_2e^{-(gx_1.gx_2)_g}$$
By Lemma \ref{approx} $e^{-(gx_1.gx_2)_g}\leq e^{26\delta +12}e^{d(\ast, g)}e^{-(gx_1.gx_2)_\ast}$, and since $d_V$ is visual, $e^{-(gx_1.gx_2)_\ast}\leq {1\over k_1}d_L(x_1,x_2)$. Combining, $k_2e^{-(gx_1.gx_2)_g}\leq Ke^{d(\ast,g)}d_L(x_1,x_2)$, so:
$$D_{gC(x_1,x_2)}\leq K^2q_Ad_L(x_1,x_2)$$
\end{proof}

\section{Equivalence of the Two Metrics}\label{dV=dL}




\begin{theorem}\label{cont} 
The metrics $d_V$ and $d_L$ define the same topology on  $\partial X$.
\end{theorem}
\begin{proof} 
Since $\partial X$ is compact with the metric $d_V$, it suffices to show the identity map from the $d_V$ metric to the $d_L$ metric is continuous. First a brief outline of the proof. 
We will show there is a constant $N$ such that $d_L(x_1,x_2)\leq N(d_V(x_1,x_2))^{1\over 4}$ for $x_1,x_2\in \partial X$.  This implies that for a given $\epsilon>0$ and $\delta=({\epsilon\over N})^4$, if $d_V(x_1,x_2)<\delta$ then $d_L(x_1,x_2)<\epsilon$.  For certain cut points $c_{-1},c_0$ and $c_1$ of $\partial X$ (each of which separate $x_1$ from $x_2$) and in several different situations, we produce functions of $d_V(x_1,x_2)$  that bound  $d(x_2,c_1)$ (Lemma \ref{C2.1}), $d_L(c_1,c_0)$ (Lemmas \ref{C1} and Lemma \ref{C2}), $d_L(c_{0},c_{-1})$ (Lemma \ref{C3}), and $d_L(c_{-1}, x_1)$ (Lemma \ref{C2.1}). Combining these results with the triangle inequality produces the desired inequality. 

Suppose $x_1, x_2\in \partial X$ and $C_{x_1,x_2}=\{\ldots, c_{-1},c_0,c_1\dots\}$ is the ordered set of cut points separating $x_1$ from $x_2$ in $X$. For $i\in \{1,2\}$, let $r_i$ be a geodesic ray from $\ast\in X$ to $x_i$ and $l$ be a geodesic line in $X$ with ends converging to $x_1$ and $x_2$. Consider $\hat r_i$ and 
$\hat l$ in $\mathcal T$  as in Lemma \ref{Tree}. Then $\hat r_1=(e_1,e_2,\ldots e_n, f_{-1}, f_{-2},\ldots )$, $\hat r_2=(e_1,\ldots, e_n, f_{1},f_{2},\ldots)$ and 
$\hat l$ has the form $(\ldots, f_{-1}, f_0, f_1, \ldots)$. For $i\in\{1,2,\ldots, n\}$ let $p_i$ be a closest vertex of $p_i C=e_i$ to $\ast$ where the limit set of $p_iC$ is $\{d_i\}$. 
For  $i\in\{\ldots, -1,0,1,\ldots \}$ let $q_i$ be a closest vertex of $q_i C=f_i$ to $\ast$ where the limit set of $q_iC$ is $\{c_i\}$.

In $X$:

(1) For $i\in\{1,2,\ldots, n\}$, the coset $p_iC$ separates 

\noindent $\{C, p_1C,\ldots p_{i-1} C\}$ from $\{p_{i+1}C, \ldots , p_nC\}\cup \{\ldots, q_{-1}C, q_0C, q_1C\ldots \}$. 
{\vskip .25in}
(2) If $i<0$, the coset $q_iC$ separates 

\noindent $\{\ldots, q_{i-2}C, q_{i-1}C\}$ from  $\{q_{i+1}C, q_{i+2}C,\ldots  \} \cup \{C, p_{1}C, \ldots , p_nC\}$.
{\vskip .25in}
(3) If $i\geq 0$, the coset $q_iC$ separates 

\noindent $\{ q_{i+1}C, q_{i+2}C\ldots\}$ from  $\{q_{i-1}C, q_{i-2}C,\ldots  \}\cup \{C, p_{1}C, \ldots , p_nC\}$.
{\vskip .25in}

Note that even when $n=0$, so that  $\ast$, $q_0C$ and $q_{-1}C$ belong to the same vertex group, the following hold true. 

(4) The coset $f_0=q_0C$ either contains $\ast$ or separates $\ast$ from $\{q_1C,q_2C,\ldots \}$. In particular $r_2$ contains a point of $q_iC$ for all $i\geq 0$
{\vskip .25in}

(5) The coset $f_{-1}=q_{-1}C$ either contains $\ast$ or separates $\ast$ from $\{q_{-2}C,q_{-3}C,\ldots \}$. In particular $r_1$ contains a point of $q_{i}C$ for all $i<0$.

{\vskip .25in}
Lemma \ref{close} implies:

\begin{lemma}\label{si}
For each $i\geq 0$ the geodesic $r_2$ contains a point of $q_iC$ and if $s_i$ is the first point in $[0,\infty)$ such that $r(s_i)\in q_iC$, then $d(r_2(s_i),q_i)\leq 6\delta +4$. If $i<0$ the geodesic $r_1$ contains a point of $q_iC$ and if $s_i$ is the first point in $[0,\infty)$ such that $r_1(s_i)\in q_iC$, then $d(r_1(s_i),q_i)\leq 6\delta +4$.
\end{lemma}

Let $z_1, z_2, z_3$ be internal points of the ideal geodesic triangle with sides $r_1, r_2, l$, where $z_1$ is on $r_1$, $z_2$ on $r_2$ and $z_3$ on $l$. 
Then: 
$$(x_1. x_2)_\ast =d(\ast, z_1)=d(\ast, z_2)$$ 

 
 
Let $m={(x_1. x_2)_\ast\over 2}$ so that $r_2(m)$ is half way between $\ast$ and $z_2$ on $r_2$. 
We consider two cases. When $s_0\geq m$ and $s_0<m$. The following constant appears many times in what follows, where $D$ is the larger of $D_A$ and $D_B$ (the respective diameters of $A$ and $B$, with respect to $d_V$).
$$Q=k_2k_1^{-{3\over 2}}e^{32\delta+16} D$$

\begin{lemma} \label{C1} 
Suppose $s_0\geq m$, then:
$$(1)\ \ \ d_V(c_0,c_1)\leq Q \sqrt{d_V(x_1,x_2)}\hbox{ and }$$
$$(2)\ \ \ \ \ \ \ \  d_L(c_0,x_2)\leq 2Q\sqrt{d_V(x_1,x_2)}$$
\end{lemma} 
\begin{proof}
In this case, 
$$d(\ast, q_0)+ (6\delta-4)\geq s_0\geq m$$ 
Without loss, assume that $q_0C$ and $q_1C$ bound $q_0A$ - as opposed to $q_0B$. By Lemma \ref{compare}:
$$D_{q_0A}\leq {k_2\over k_1}e^{(26\delta+12)-d(\ast, q_0)}D_A\leq {k_2\over k_1}e^{26\delta +12-m +6\delta+4}D_A={k_2\over k_1}e^{32\delta+16}D_A e^{-{(x_1. x_2)_\ast\over 2}}$$ 
Combining the last inequality with  $e^{-(x_1.x_2)_\ast} \leq {1\over k_1}d_V(x_1,x_2)$ gives:
$$d_V(c_0,c_1)\leq D_{q_0A}\leq {k_2\over k_1}e^{32\delta+16} D_A \sqrt{d_V(x_1,x_2)\over k_1}=Q\sqrt{d_V(x_1,x_2)}$$
So Equation (1) is established.

Observing that $d(\ast, q_1)\geq d(\ast, q_0)+1$ and applying Lemma \ref{compare} implies:
$$d_V(c_1,c_2)\leq D_{q_1B}\leq {k_2\over k_1}e^{(26\delta+12)-d(\ast, q_1)}D_B\leq {k_2\over k_1}e^{-1}e^{32\delta+16}D_B e^{-{(x_1. x_2)_\ast\over 2}}$$ 
In general, we have for $i\geq 0$, $d(\ast, q_i)\geq d(\ast, q_0)+i$ and so:
$$d_V(c_i,c_{i+1})\leq {k_2\over k_1}e^{-i}e^{32\delta+16} D e^{-{(x_1. x_2)_\ast\over 2}}\leq {k_2\over k_1}e^{-i}e^{32\delta+16} D \sqrt{d_V(x_1,x_2)\over k_1} \hbox{ and}$$ 
$$d_V(c_i,c_{i+1})\leq e^{-i} Q\sqrt{d_V(x_1,x_2)}$$
Since $\sum_{i=0}^\infty e^{-i}\leq\sum_{i=0}^\infty 2^{-i}= 2$:
$$d_L(c_0,x_2)=\sum_{i=1}^\infty d_V(c_{i-1}, c_i) \leq Q\sqrt{d_V(x_1,x_2)}\sum_{i=0}^\infty e^{-i}$$
$$d_L(c_0,x_2)\leq 2Q\sqrt{d_V(x_1,x_2)}$$
\end{proof}

Now we consider Case 2.

Recall that  $r_1(2m)=z_1$, $r_2(2m)=z_2$ and $\{z_1,z_2,z_3\}$ are ideal points of the ideal geodesic triangle $\triangle (r_1,r_2,l)$. For $i\geq 0$, let $s_i$ be the first point of $[0,\infty)$ such that $r_2(s_i)\in q_iC$. If $i<0$, let $s_i$ be the first point of $[0,\infty)$ such that $r_1(s_i)\in q_iC$.  

\begin{lemma}\label{sm} 
Suppose $s_0<m$ and $m\geq 14\delta$, then:  

\medskip

\noindent $(i)$ Let ${\bf H_i}$ be the horoball for $q_iC$. If  $i\in\{-1,0,1\}$ and $s_i<2m$, then: 

$$max \{d(z_2,{\bf H_i}),d(z_1, {\bf H_i})\leq 4\delta.$$ 

\noindent $(ii)$ For $i\in \{-1,1\}$: 
$$s_{i}> 2m-14\delta=m+(m-14\delta)\geq m.$$

\end{lemma}

\begin{proof} 
First we prove ($i$). Lemma \ref{close} implies: 
$$max\{d(q_1, r_2(s_1)), d(q_0,r_2(s_0)),d(q_{-1}, r_1(s_{-1}))\}\leq 6\delta +4$$ 
For $i\in \{-1,0,1\}$, let $w_i$ be a point of $q_iC$ on the line $l$ (Lemma \ref{cut1}). Consider the ideal triangle $\triangle (r_1, r_2, l)$ (see Figure 11). 

{\it By Remark \ref{sub1}, we assume the ideal triangle $\triangle (r_1, r_2, l)$ is $\delta$ (not $5\delta$) thin.}

If $w_i$ is between $z_3$ and $x_2$, then the point $w_i'$ of $r_2$ corresponding to $w_i$ is within $\delta$ of $q_iC$.  Lemma \ref{QC} (applied to $w_0'$, $z_2$ and $r_2(s_0)$ when $i=0$, to $w_1'$, $z_2$ and $r_2(s_1)$ when $i=1$, and to $w_{-1}'$, $z_2$ and $r_2(s_{-1})$ when $i=-1$)  implies $z_2$ is within $3\delta$ of  ${\bf H_i}$ and so $z_1$ is within $4\delta$ of ${\bf H_i}$. 
If $w_i$ is between $z_3$ and $x_1$, let $w_i'$ be the corresponding point of $r_1$ (within $\delta$ of $w_i$). Lemma \ref{QC} (applied to the $r_1$ points $w_0'$, $z_1$ and $r_1(s_0)$ when $i=0$,  to $w_1'$, $z_1$ and $r_1(s_1)$ when $i=1$ and to $w_{-1}'$, $z_1$ and $r_{1}( s_{-1})$ when $i=-1$) implies $z_1$ is within $3\delta$ of  ${\bf H_i}$ and so $z_2$ is within $4\delta$ of ${\bf H_i}$. 
In any case,  $z_1$ and $z_2$ are within $4\delta$ of ${\bf H_i}$ for $i\in \{-1,0,1\}$ (so part ($i$) is proved). 


\vspace {.6in}
\vbox to 3in{\vspace {-2in} \hspace {-1.7in}
\includegraphics[scale=1]{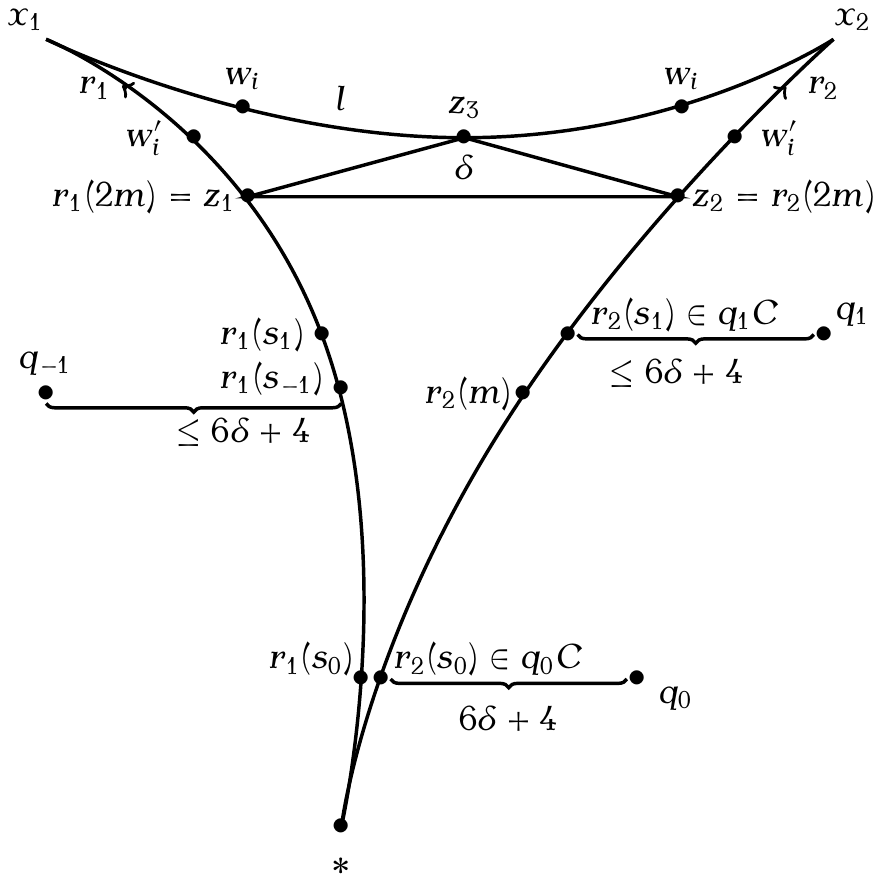}
\vss }
\vspace{.4in}

\centerline{Figure 11}

\medskip

Lemma \ref{Deep} applied to ${\bf H_0}$ and $r_2|_{[s_0,2m]}$, ${\bf H_1}$ and $r_2|_{[s_1,2m]}$, and ${\bf H_{-1}}$ and $r_1|_{[s_{-1},2m]}$ respectively (with $N=4\delta$), implies:
$$(1)\ \ \ \ \ \ r_2([s_0+7\delta, 2m-7\delta])\subset {\bf H_0^\delta}$$ 
$$(2)\ \ \ \ \ \ r_2([s_1+7\delta, 2m-7\delta])\subset {\bf H_1^\delta}$$
$$(3)\ \ \ r_1([s_{-1}+7\delta, 2m-7\delta])\subset {\bf H_{-1}^\delta}$$

Since ${\bf H_0^\delta}\cap {\bf H_1^\delta}=\emptyset$ we have: 
$$[s_0+7\delta, 2m-7\delta]\cap [s_1+7\delta, 2m-7\delta]=\emptyset$$
By hypothesis, $m\geq 14\delta$. As $s_0<m$, $m+(m-s_0)> 14\delta$. Then $2m> 14\delta+s_0$ and $(2m-7\delta)-(s_0+7\delta)> 0$ implying $[s_0+7\delta, 2m-7\delta]\ne \emptyset$. Since $[s_0+7\delta, 2m-7\delta]$ and $[s_1+7\delta, 2m-7\delta]$ have the same right end point and empty intersection, it must be that $[s_1+7\delta, 2m-7\delta]=\emptyset$, equivalently $(2m-7\delta)-(s_1+7\delta)<0$. Then $s_1>2m-14\delta$. We have verified part $(ii)$ when $i=1$:
$$(4)\ \ \ \ \ m\geq 14\delta \hbox { implies } s_1>2m-14=m+(m-14\delta)\geq m$$

Since $d({\bf H_0^\delta}, {\bf H_{-1}^\delta})>\delta$ and $d(r_1(t), r_2(t))\leq \delta$ for all  $t\in [0,2m]$ equations $(1)$ and $(3)$ imply:
$$[s_0+7\delta, 2m-7\delta]\cap [s_{-1}+7\delta, 2m-7\delta]=\emptyset$$ 
Again, since $m\geq 14\delta$,  $[s_0+7\delta, 2m-7\delta]\ne \emptyset$ and so $[s_{-1}+7\delta, 2m-7\delta]=\emptyset$. Equivalently, $s_{-1}>2m-14\delta$. 
$$(4')\ \ \ \ \ m\geq 14\delta \hbox { implies } s_{-1}>2m-14\delta\geq m$$
Equations $(4)$ and $(4')$ verify part $(ii)$ of the Lemma.
\end{proof}

\begin{lemma} \label{C2.1} 
Suppose $s_0<m$ and $m>14\delta$, then: 
$$max\{d_L(c_{-1},x_1),d_L(c_1,x_2)\}\leq 2Q\sqrt{d_V(x_1,x_2)}$$
\end{lemma}
\begin{proof}
The lemma follows immediately from Lemma \ref{sm}, and applying Lemma \ref{C1} twice - first with $s_0$ replaced by $s_1$ and $c_0$ replaced by $c_1$; and second with $s_0$ replaced by $s_{-1}$, $c_{0}$ replaced by $c_{-1}$ and $x_2$ replaced by $x_1$. 
\end{proof}

\begin{lemma} \label{C2} 
Suppose $s_0<m$ and $m\geq 54\delta+24+J_{\ref{Deep}}(40\delta+20)$ then 
there is a constant $M_1$ such that:
$$d_V(c_0,c_1)\leq M_1(d_V(x_1,x_2))^{1\over 4}$$
\end{lemma}

\begin{proof}
Our goal is to show that a geodesic from $q_0$ to $c_0$ and a geodesic from $q_0$ to $c_1$ will $\delta$-fellow travel for a ``long" distance (depending on $m$). This is equivalent to $(c_0. c_1)_{q_0}$ being ``large". A geodesic from $q_0$ to $c_0$ is the vertical geodesic at $q_0$. Let $q_1'$ be a closest point of $q_1C$ to $q_0$. A geodesic from $q_0$ to  $q_1'$  followed by a vertical geodesic is a geodesic from $q_0$ to $c_1$.  
\vskip .15in

\begin{claim}\label{cl1} 
If $m>14\delta$ and $s_0<m$, then the following inequalities hold:


$(i)\ \ \ \  d(q_1', r_2(s_1))\leq 34\delta+20$

$(ii)\ \ \  s_1-s_0>m-14\delta$  
\end{claim}
\begin{proof} By Lemma \ref{sm}(ii), $s_1>m$, so $[s_0,s_1]$ is non-empty.  By Lemma \ref{close}, a geodesic from $q_0$ to $q_1$ contains a point $q_1''\in q_1C$ such that $d(q_1', q_1'')\leq 6\delta+4$.   See Figure 12.
Since $d(q_1, r_2(s_1))\leq 6\delta +4$ and $d(q_0, r_2(s_0))\leq 6\delta+4$, Lemma \ref{parallel} implies there is $n\in [s_0,s_1]$ such that:
$$d(r_2(n), q_1'')\leq 8\delta+4$$
By the triangle inequality:
$$n-14\delta-8\leq d(q_1',\ast)\leq n+14\delta +8$$
$$s_1-6\delta-4\leq d(q_1,\ast)\leq s_1+6\delta+4$$

\vspace {.6in}
\vbox to 3in{\vspace {-2in} \hspace {-1.8in}
\includegraphics[scale=1]{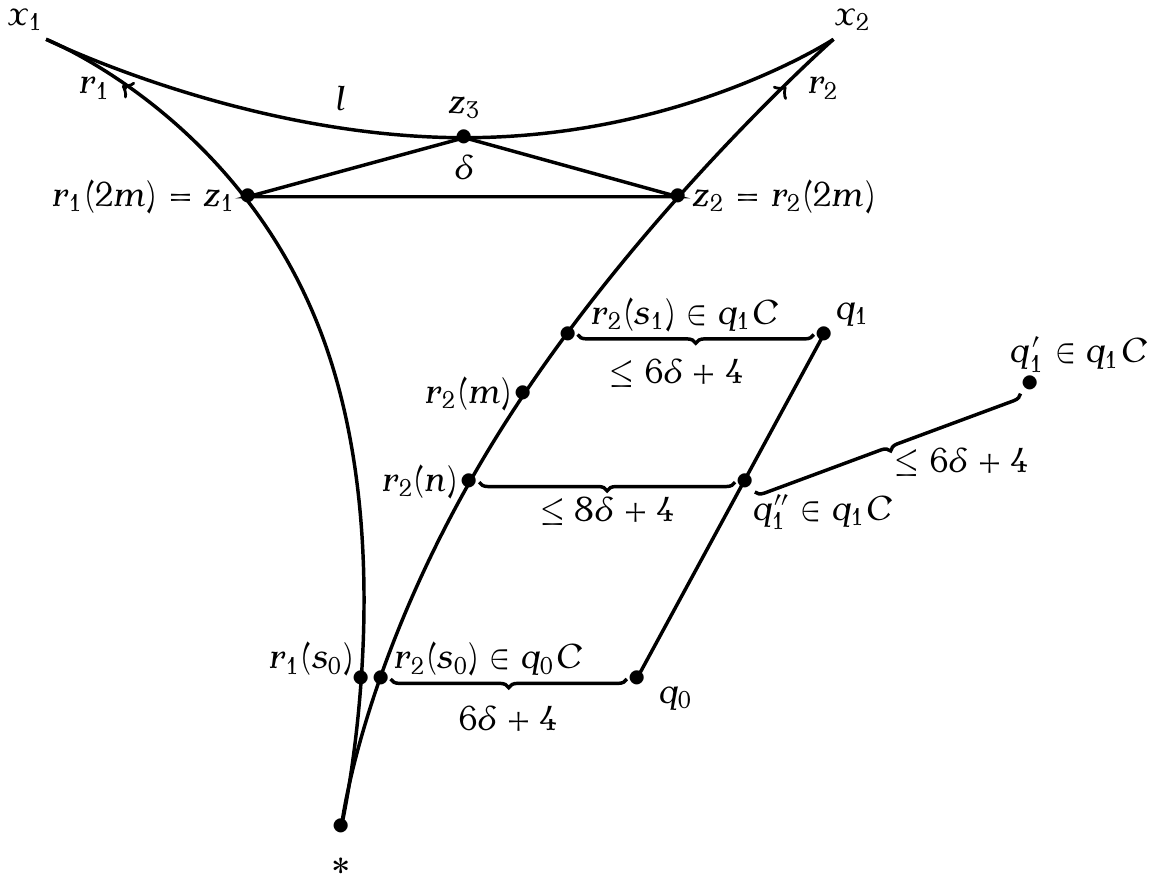}
\vss }
\vspace{.4in}

\centerline{Figure 12}

\medskip

Recall, $q_1'\in q_1C$ where $q_1$ is a closest point of $q_1C$ to $\ast$. Combining the last two inequalities:
$$0\leq d(q_1',\ast)-d(q_1,\ast)\leq n+14\delta+8-(s_1-6\delta-4)=(20\delta+12)-(s_1-n)$$
$$ 0\leq s_1-n\leq 20\delta+12$$
$$d(q_1',r_2(s_1))\leq d(q_1',q_1'')+d(q_1'', r_2(n))+ d(r_2(n), r_2(s_1))$$ 

This implies:
$$(i)\ \ \ \ \ d(q_1', r_2(s_1))\leq (6\delta+4) +(8\delta+4)+(20\delta+12)=34\delta+20$$
Lemma \ref{sm}$(ii)$ implies $s_1>2m-14\delta$. Since $m>s_0$:
$$(ii)\ \ \ \ \ s_1-s_0>s_1-m>m-14\delta$$
This completes the proof of the Claim. 
\end{proof}

At this point we consider two cases. The first case is when $s_1\geq 2m$ the second (more complicated case) is when $s_1<2m$. 
\vskip .15in
\noindent {\bf Case 1.} Assume that $s_1\geq 2m$.
\vskip.15in

Let $\alpha$ be a geodesic from $q_0$ to $q_1'$. We want to show there is a ``large" integer $k$ such $\alpha(k)$ is ``close" to ${\bf H_0}$ (the horoball over $q_0C$) and then use Lemma \ref{Deep} to show an initial segment of $\alpha$ can be replaced by a geodesic with a ``long" vertical initial segment in ${\bf H_0}$. This allows us to show that $(c_0.c_1)_{q_0}$ is ``large" when $m$ is ``large".

By Claim \ref{cl1}$(i)$, $d(q_1', r_2(s_1))\leq 34\delta+20$ and by Lemma \ref{parallel} there is an integer $k$ such that $d(\alpha(k),z_2)\leq 36\delta+20$ (see Figure 13). 

\vspace {.4in}
\vbox to 3in{\vspace {-2in} \hspace {-1.8in}
\includegraphics[scale=1]{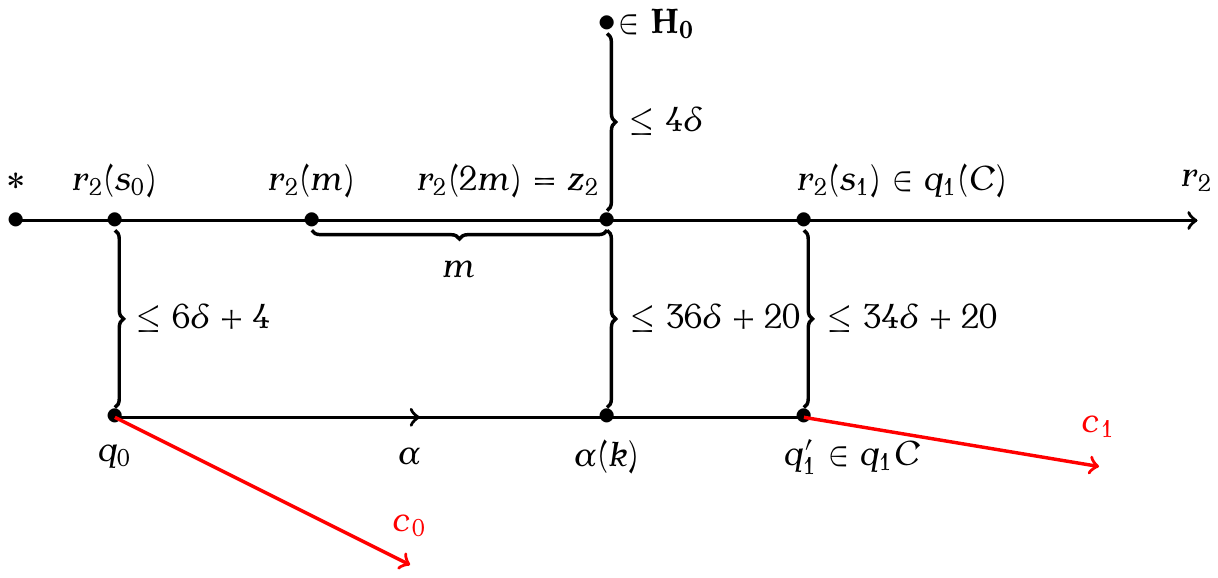}
\vss }
\vspace{-1in}

\centerline{Figure 13}

\medskip

Since $s_0<m$, Lemma \ref{sm}$(i)$ implies $d(z_2,{\bf H_0})\leq  4\delta$. By the triangle inequality, we have $\alpha(k)$ ``close" to ${\bf H_0}$ :
$$d(\alpha(k),{\bf H_0})\leq d(\alpha(k), z_2)+d(z_2,{\bf H_0})\leq 40\delta+20$$ 

Since $s_0<m$ and $z_2=r_2(2m)$, we have $d(r_2(s_0),z_2)>m$, and we see that $k$ is ``large" when $m$ is ``large":
$$m<d(r_2(s_0),z_2)\leq d(r_2(s_0),q_0)+d(q_0, \alpha(k))+d(\alpha(k), z_2)$$
$$k=d(q_0,\alpha(k))>m-(6\delta+4)-(36\delta +20)=m-42\delta-24\geq J_{\ref{Deep}}(40\delta+20)$$

Since $\alpha(0)\in q_0C\subset {\bf H_0}$ and $d(\alpha(k), {\bf H_0})\leq 40\delta+20$, Lemma \ref{Deep} implies that there is a geodesic $\alpha'$, from $q_0$ to $\alpha(k)$ (and hence from $q_0$ to $q_1'$) with initial vertical segment of length $$M={1\over 2}(k-43\delta-20)>{1\over 2} (m-85\delta-44)$$ 
In particular, the vertical geodesic in ${\bf H_0}$ at $q_0$ (converging to $c_0$) and the geodesic $\alpha'$ followed by the vertical geodesic at $q_1'$ in ${\bf H_1}$ have an initial overlap of length $M$. Then $(c_0. c_1)_{q_0}\geq M$.  Let 
$$L_1={1\over 2} (85\delta+44)$$

\noindent {\bf Case 2.} Assume that $s_1<2m$. 
 $$s_1-s_0=d(r_2(s_0), r_2(s_1))\leq d(r_2(s_0), q_0)+d(q_0,q_1')+d(q_1', r_2(s_1))$$
$$s_1-s_0\leq6\delta+4+d(q_0, q_1')+34\delta+20=d(q_0,q_1')+40\delta+24$$
By part Claim \ref{cl1}$(ii)$, $s_1-s_0>m-14\delta$, so:
$$d(q_0, q_1')\geq s_1-s_0-40\delta-24> m-54\delta-24\geq J_{\ref{Deep}}(40\delta+20)$$
Since $r_2(s_0)\in q_0C$ and $d(z_2,{\bf H_0})\leq 4\delta$ (see Lemma \ref{sm}$(i)$), Lemma \ref{QC} implies $r_2(s_1)$ is within $6\delta$ of $q_0(C)$. By Claim \ref{cl1}$(i)$, $d(q_1',r_2(s_1))\leq 34\delta+20$ and so $q_1'$ is within $40\delta +20$ of $q_0C$.

Since $d(q_0, q_1')\geq J_{\ref{Deep}}(40\delta+20)$ Lemma \ref{Deep} implies there is a geodesic from $q_0$ to $q_1'$ with an initial vertical segment of length: 
$$M'={1\over 2} (d(q_0,q_1')- (43\delta+20))>{1\over 2}(m-54\delta-24-(43\delta+20))={1\over 2}(m-97\delta-44)$$ 

In particular, $(c_0. c_1)_{q_0}$ is at least as large as this last number. Let 
$$L_2={1\over 2}(97\delta+44)$$ 
so that $L_2>L_1$. In either case: 
$$(c_0. c_1)_{q_0} \geq {m\over 2}- L_2={(x_1. x_2)_\ast\over 4}- L_2$$
$$e^{-(c_0. c_1)_{q_0}} \leq e^{-(x_1. x_2)_\ast\over 4} e^{L_2}\leq e^{L_2} k_1^{-{1\over 4}} (d_V(x_1,x_2))^{1\over 4}$$
By Lemma \ref{approx}:
$$d_V(c_0,c_1)\leq k_2e^{-(c_0. c_1)_\ast}\leq k_2e^{-d(\ast,q_0)} e^{-(c_0. c_1)_{q_0}} e^{26\delta+12}\leq k_2e^{-(c_0. c_1)_{q_0}} e^{26\delta+12}$$
Combining these last two inequalities and letting $M_1=k_2k_1^{-{1\over 4}}e^{L_s+26\delta+12}$:
$$(5)\ \ \ d_V(c_0,c_1)\leq  (d_V(x_1,x_2))^{1\over 4}k_2k_1^{-{1\over 4}}e^{L_2+26\delta+12}=M_1(d_V(x_1,x_2))^{1\over 4}$$
This completes the proof of Lemma \ref{C2}.
\end{proof}

\hskip .15in
\begin{lemma} \label{C3} 
Suppose $s_0<m$ and $m>57\delta+24+J_{\ref{Deep}}(40\delta+20)$, then 
there is a constant $M_{-1}$ such that:
$$d_V(c_{-1},c_0)\leq M_{-1}(d_V(x_1,x_2))^{1\over 4}$$
\end{lemma}
\begin{proof} 
The argument is completely similar to the one bounding $d_L(c_0,c_1)$ in Lemma \ref{C2}. The fact that $d(r_1(s_0), q_0)\leq 7\delta+4$ (as opposed to  $d(r_2(s_0),q_0)\leq 6\delta+4$) increases our bounds in an elementary way. 
  
Again, our goal is to show that the geodesic from $q_0$ to $c_0$ and the geodesic from $q_0$ to $c_{-1}$ will $\delta$-fellow travel for a distance depending on $m$. This is equivalent to $(c_0. c_{-1})_q$ being ``large" when $m$ is large. A geodesic from $q_0$ to $c_0$ is the vertical geodesic at $q_0$. Let $q_{-1}'$ be a closest point of $q_{-1}C$ to $q_0$. A geodesic from $q_0$ to  $q_{-1}'$  followed by a vertical geodesic is a geodesic from $q_0$ to $c_{-1}$.  

\begin{claim}\label{cl2} 
If $m>14\delta$, then the following inequalities hold:


$(i)\ \ \ \  d(q_{-1}', r_1(s_{-1}))\leq 36\delta+20$

$(ii)\ \ \  s_{-1}-s_0>m-14\delta$  
\end{claim}

\begin{proof}
By Lemma \ref{close}, a geodesic from $q_0$ to $q_{-1}$ contains a point $q_{-1}''\in q_1C$ such that $d(q_{-1}', q_{-1}'')\leq 6\delta+4$.   See Figure 14.

Since $d(q_{-1}, r_1(s_{-1}))\leq 6\delta +4$ and $d(q_0, r_1(s_0))\leq 7\delta+4$ Lemma \ref{parallel} implies there is $n'\in [s_0,s_1]$ such that:
$$d(r_1(n'), q_{-1}'')\leq 9\delta+4$$

\vspace {.6in}
\vbox to 3in{\vspace {-2in} \hspace {-1.8in}
\includegraphics[scale=1]{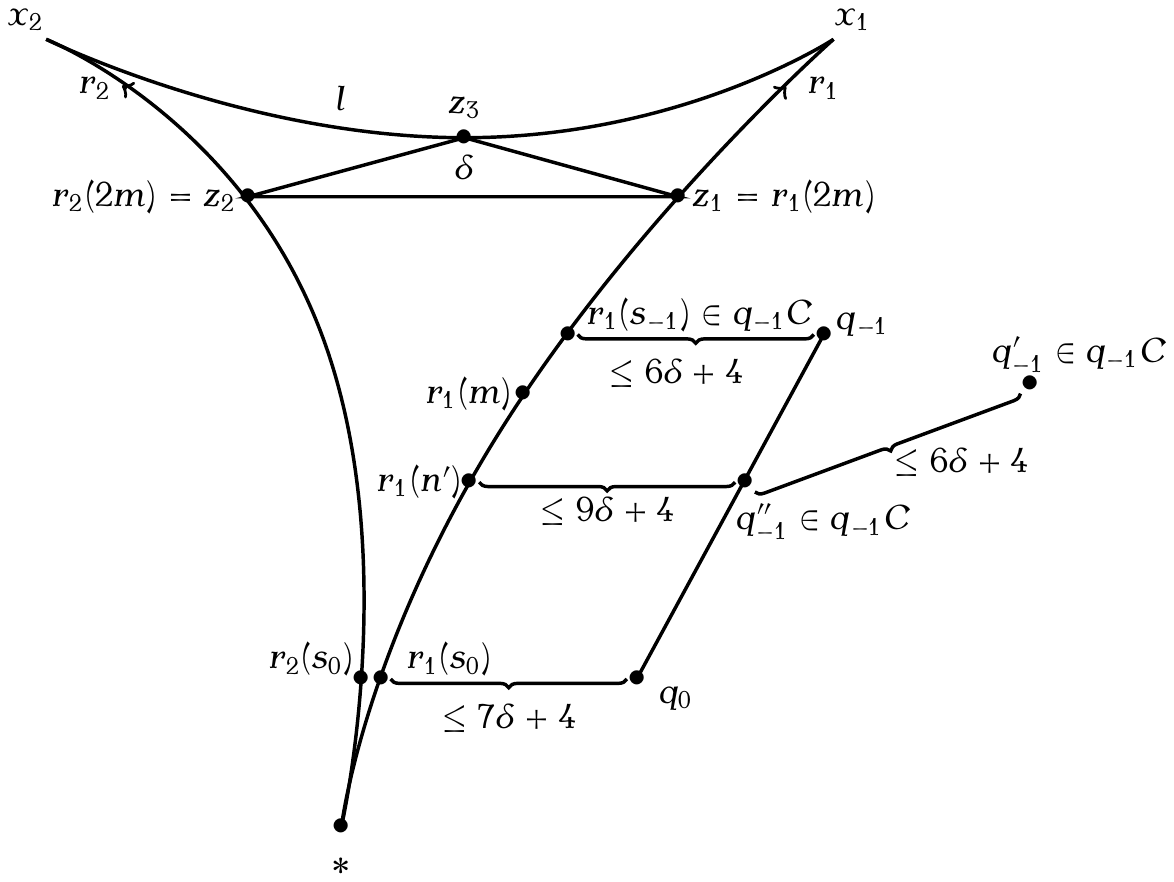}
\vss }
\vspace{.4in}

\centerline{Figure 14}

\medskip

By the triangle inequality:
$$n'-15\delta-8\leq d(q_{-1}',\ast)\leq n'+15\delta +8$$
$$s_{-1}-6\delta-4\leq d(q_{-1},\ast)\leq s_{-1}+6\delta+4$$
Combining:
$$0\leq d(q_{-1}',\ast)-d(q_{-1},\ast)\leq n'+15\delta+8-(s_{-1}-6\delta-4)=(21\delta+12)-(s_{-1}-n')$$
$$ 0\leq s_{-1}-n'\leq 21\delta+12$$
$$d(q_{-1}',r_1(s_1))\leq d(q_{-1}',q_{-1}'')+d(q_{-1}'', r_1(n'))+ d(r_1(n'), r_1(s_{-1}))$$ 
This implies:
$$(i)\ \ \ \ \ d(q_{-1}', r_1(s_{-1}))\leq (6\delta +4)+(9\delta +4)+(21\delta+12)=36\delta+20$$
Since $s_0<m$ and (by Lemma \ref{sm}$(ii)$) $s_{-1}>2m-14\delta$
$$(ii)\ \ \ \ \ s_{-1}-s_0>s_{-1}-m>m-14\delta$$
This completes the proof of the Claim.
\end{proof}

Again we consider two cases. The first case is when $s_{-1}\geq 2m$ the second  is when $s_1<2m$. 

\noindent {\bf Case A.} Assume that $s_{-1}\geq 2m$.

Let $\alpha$ be a geodesic from $q_0$ to $q_{-1}'$. By Lemma \ref{cl2}$(i)$, $d(q_{-1}', r_1(s_{-1}))\leq 36\delta+20$ and by Lemma \ref{parallel} there is an integer $k$ such that 
$$d(\alpha(k),z_1)\leq 38\delta+20$$ 
Since $s_0<m$, Lemma \ref{sm}$(i)$ implies, $d(z_1,{\bf H_0})\leq  4\delta$. By the triangle inequality:
$$d(\alpha(k),{\bf H_0})\leq d(\alpha (k), z_1)+d(z_1, {\bf H_0})\leq 42\delta+20$$ 
Since $s_0<m$ and $z_1=r_1(2m)$, we have $d(r_1(s_0),z_1)>m$ and:
$$m<d(r_1(s_0),z_1)\leq d(r_1(s_0),q_0)+d(q_0, \alpha(k))+d(\alpha(k), z_1)$$
$$k=d(q_0,\alpha(k))>m-(7\delta+4)-(38\delta +20)=m-45\delta-24\geq J_{\ref{Deep}}(42\delta+20)$$ 

Since $\alpha(0)\in q_0C\subset {\bf H_0}$ and $d(\alpha(k),{\bf H_0})\leq 42\delta+20$, Lemma \ref{Deep} implies that (since  $k\geq J_{\ref{Deep}}(42\delta+20)$) there is a geodesic $\alpha'$, from $q_0$ to $\alpha(k)$ (and hence from $q_0$ to $q_{-1}'$) with initial vertical segment in ${\bf H_0}$ of length 
$$M'={1\over 2}(k-(45\delta+20))\geq{1\over 2}(m-(45\delta+24)-(45\delta +20))={1\over 2}(m-90\delta-44)$$  
In particular the vertical geodesic at $q_0$ in ${\bf H_0}$ (converging to $c_0$) and the geodesic $\alpha'$ followed by the vertical geodesic at $q_{-1}$ in ${\bf H_{-1}}$ have initial overlap of length $M'$. Thus  $(c_0. c_{-1})_{q_0}\geq M'$.  Let 
$$L_{-1}={1\over 2} (90\delta+44)$$

\noindent {\bf Case B.} Assume that $s_{-1}<2m$. 
 $$s_{-1}-s_0=d(r_1(s_0), r_1(s_{-1}))\leq d(r_1(s_0), q_0)+d(q_0,q_{-1}')+d(q_{-1}', r_1(s_{-1}))$$
$$s_{-1}-s_0\leq7\delta+4+d(q_0, q_{-1}')+36\delta+20=d(q_0,q_{-1}')+43\delta+24$$
By Claim \ref{cl2}$(ii)$ $s_{-1}-s_0>m-14\delta$ so:
$$d(q_0, q_{-1}')\geq s_{-1}-s_0 -43\delta -24> m-57\delta-24\geq J_{\ref{Deep}}(42\delta+20)$$
Since $d(r_1(s_0),r_2(s_0))\leq \delta$,  $r_2(s_0)\in q_0C$ and $d(z_1,{\bf H_0})\leq 4\delta$ (by Lemma \ref{sm}$(i)$),  Lemma \ref{QC} implies $r_1(s_{-1})$ is within $6\delta$ of $q_0C$. By Claim \ref{cl2}$(i)$, $d(q_{-1}',r_1(s_{-1}))\leq 36\delta+20$ and so $q_{-1}'$ is within $42\delta +20$ of $q_0C$.

Since $d(q_0, q_{-1}')\geq J_{\ref{Deep}}(42\delta+20)$, Lemma \ref{Deep} implies that there is a geodesic from $q_0$ to $q_{-1}'$ with an initial vertical segment in ${\bf H_0}$ of length: $$M_1'={1\over 2} (d(q_0,q_{-1}')-(45\delta+20))\geq {1\over 2}(m-104\delta-44)$$
 In particular, $(c_0. c_{-1})_{q_0}$ is at least as large as this last number. Let 
$$L_{-2}={1\over 2}(104\delta+44)>L_{-1}$$ 
In either case: 
$$(c_0. c_{-1})_{q_0} \geq {m\over 2}- L_{-2}={(x_1. x_2)_\ast\over 4}- L_{-2}$$
$$e^{-(c_0. c_{-1})_{q_0}} \leq e^{-(x_1. x_2)_\ast\over 4} e^{L_{-2}}=k_1^{-{1\over 4}}e^{L_{-2}} (d_V(x_1,x_2))^{1\over 4}$$
By Lemma \ref{approx}:
$$d_V(c_0,c_{-1})\leq k_2e^{-(c_0. c_{-1})_\ast}\leq k_2e^{-d(\ast,q_0)} e^{-(c_0. c_{-1})_{q_0}} e^{26\delta+12}\leq k_2e^{-(c_0. c_{-1})_{q_0}} e^{26\delta+12}$$
Combining these last two inequalities and letting $M_{-1}=k_2k_1^{-{1\over 4}}e^{L_{-2}+26\delta+12}$:
$$d_V(c_0,c_{-1})\leq  (d_V(x_1,x_2))^{1\over 4}k_2k_1^{-{1\over 4}}e^{L_{-2}+26\delta+12}=M_{-1}(d_V(x_1,x_2))^{1\over 4}$$
This completes the proof of Lemma \ref{C3}.
\end{proof}

\begin{lemma}\label {d01} 
Suppose $s_0\geq m$ and $s_1\geq m$ then $(c_0,c_{-1})_\ast\geq m-9\delta -4$. 
\end{lemma}

\begin{proof} 
Consider a geodesic ray $t_{-1}$ from $\ast$ to $q_{-1}$ followed by the vertical geodesic in ${\bf H_{-1}}$ beginning at $q_{-1}$. This ray is geodesic  since $q_{-1}$ is a closest point of $q_{-1}C$ to $\ast$ and $t_{-1}$ converges to $c_{-1}$. See Figure 15.  Let $t_0$ be a geodesic ray from $\ast$ to $q_0$ followed by the vertical geodesic in ${\bf H_{0}}$ beginning at $q_0$, then $t_0$ converges to $c_0$. Since $d(q_0, r_2(s_0))\leq 6\delta +4$ and $r_2(s_0)\geq m$, 
$$d(q_0,\ast)\geq m-6\delta-4$$ 
Similarly,
$$ d(q_{-1},\ast)\geq m-6\delta -4$$

\vspace {.6in}
\vbox to 3in{\vspace {-2in} \hspace {-1.8in}
\includegraphics[scale=1]{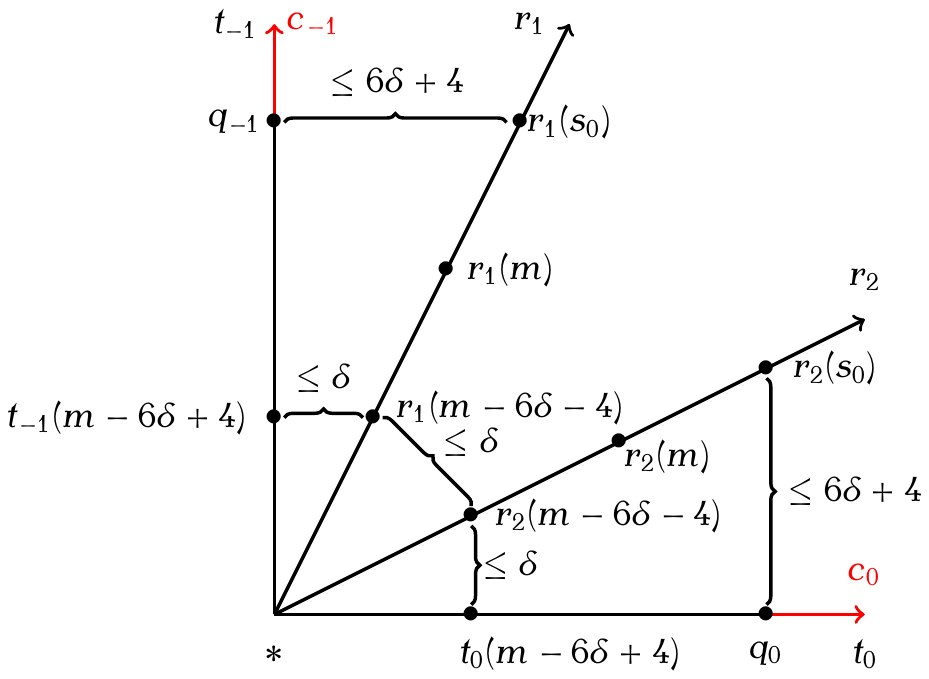}
\vss }
\vspace{-.4in}

\centerline{Figure 15}

\medskip

Considering a geodesic triangle $\triangle (\ast, q_0, r_2(s_0))$ (with one side on $t_0$, another on $r_2$ and the third of length $\leq 6\delta +4$) we have: 
$$d(t_0(m-6\delta -4), r_2(m-6\delta -4))\leq \delta$$ 
Similarly, 
$$d(t_{-1}(m-6\delta -4),r_1(m-6\delta -4))\leq \delta$$ 
Considering $\triangle (\ast, r_1(2m), r_2(2m))$ (with $d(r_1(2m),r_2(2m))\leq \delta$):
$$ d(r_1(m-6\delta -4), r_2(m-6\delta -4))\leq \delta$$
By the triangle inequality: 
$$d(t_{-1}(m-6\delta-4),t_0(m-6\delta-4))\leq 3\delta$$

Now assume that $m'=(c_0,c_{-1})_\ast<m-6\delta-4$. Let 
$$k=(m-6\delta -4)-m'$$ 
Let $l'$ be a geodesic line from $c_{-1}$ to $c_0$, so the internal points of the ideal geodesic triangle $\triangle (t_{-1}, t_0,l')$ are $t_{-1}(m')$, $t_0(m')$, and $ v$ for some vertex $v$ of $l'$. Then let $v_{-1}$ be the vertex of $l'$ (between $v$ and $c_{-1}$) such that $d(v_{-1}, v)=k$ (and so $d(t_{-1}(m-6\delta -4), v_{-1})\leq \delta$). Let $v_0$ be the vertex of $l'$ (between $v$ and $c_{0}$) such that $d(v_0, v)=k$ (and so $d(t_0(m-6\delta -4), v_0)\leq \delta$). See Figure 16.

\vspace {.6in}
\vbox to 3in{\vspace {-2in} \hspace {-.8in}
\includegraphics[scale=1]{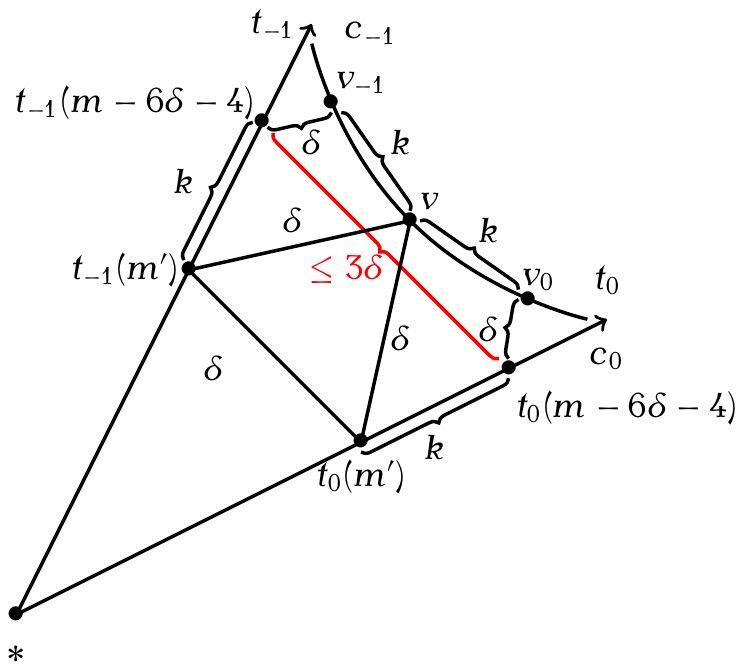}
\vss }
\vspace{-.4in}

\centerline{Figure 16}

\medskip

Then $2k=d(v_{-1},v_0)$ is less than or equal to the length of the path from $v_{-1}$ to $t_{-1}(m-6\delta-4)$ to $t_0(m-6\delta -4)$ to $v_0$. That means 
$$2k=d(v_{-1},v_0)\leq \delta +3\delta +\delta =5\delta$$ 
In particular, $k\leq {5\over 2}\delta< 3\delta$. Now 
$$3\delta>k=m-6\delta-4-m'=m-6\delta-4-(c_0.c_{-1})_\ast$$ 
implying $(c_0,c_{-1})_\ast>m-9\delta -4$.
\end{proof}
Let $M=max\{M_1,M_{-1}\}$.

By Lemmas \ref{C2} and \ref{C2.1}, if $s_0<m$ and $m>54\delta+24+J_{\ref{Deep}}(40\delta+24)$ then:
$$d_L(c_0,x_2)\leq d_L(c_0,c_1) +d_L(c_1,x_2)\leq 2Q\sqrt{d_V(x_1,x_2)} +(d_V(x_1,x_2))^{1\over 4}M$$
In our continuity argument, we may assume $d_V(x_1,x_2)\leq 1$. But then $\sqrt{d_V(x_1,x_2)} \leq (d_V(x_1,x_2))^{1\over 4}$ and:
$$d_L(c_0,x_2)\leq (2Q+M)(d_V(x_1,x_2))^{1\over 4}$$
Again by Lemmas \ref{C3} and \ref{C2.1}, if $s_0<m$ and $m>57\delta +24+J_{\ref{Deep}}(40\delta+24)$ then:
$$d_L(c_0,x_1)\leq d_L(c_0, c_{-1})+d_L(c_{-1},x_1)\leq 2Q\sqrt{d_V(x_1,x_2)} +(d_V(x_1,x_2))^{1\over 4}M$$
$$d_L(c_0,x_1)\leq (2Q+M)(d_V(x_1,x_2))^{1\over 4}$$
Combining; If $s_0<m$ and $m>57\delta +24+J_{\ref{Deep}}(40\delta+24)$ then:
$$(A)\ \ \ d_L(x_1,x_2)\leq d_L(x_1, c_0)+d_L(c_0, x_2)\leq 2(2Q+M)(d_V(x_1,x_2))^{1\over 4}$$

Two other cases must be considered. 
It may be that $s_0\geq m$ and $s_{-1}<m$ (and $m>57\delta +24+J_{\ref{Deep}}(40\delta+24)$). In this case, we simply observe that if the sets $\{c_0,c_1,\ldots\}$ and $\{c_{-1}, c_{-2},\ldots\}$ are interchanged (so that for $i\geq 0$, $c_i$ plays the role of $c_{-i-1}$)  then Lemmas \ref{C2} and \ref{C3} remain valid. In this way, equation $(A)$ (with $c_0$ replaced by $c_{-1}$) remains valid in this case. 

Finally, it may be that $s_0\geq m$ and $s_{-1}\geq m$. In this case, Lemma \ref{C1} implies that when $s_0\geq m$ then $d_L(c_0,x_2)\leq 2Q \sqrt{d_V(x_1,x_2)}$. Again interchanging the sets $\{c_0,c_1,\ldots\}$ and $\{c_{-1}, c_{-2},\ldots\}$ in Lemma \ref{C1} tells us that if $s_{-1}\geq m$ then $d_L(c_{-1},x_1)\leq 2Q\sqrt{d_V(x_1,x_2)}$. By Lemma \ref{d01} $(c_0.c_{-1})_\ast\geq m-9\delta -4={(x_1.x_2)_\ast\over 2}-9\delta-4$. Then $d_V(c_{-1}, c_0)=d_L(c_{-1}, c_0)$.
$$d_V(c_0,c_{-1}) \leq k_2 e^{-(c_0.c_{-1})_\ast} \leq k_2e^{9\delta+4} \sqrt {e^{-(x_1.x_2)_\ast}}\leq k_2 k_1^{-{1\over 2}}e^{9\delta+4} \sqrt{d_V(x_1,x_2)}$$
Combining and applying the triangle inequality:
$$d_L(x_1,x_2)\leq d_L(x_1,c_{-1}) +d_L(c_{-1},c_0)+d_L( c_0, x_2)\leq$$
$$(2Q+   k_2 k_1^{-{1\over 2}}e^{9\delta+4}   +2Q)\sqrt{d_V(x_1,x_2)}$$
Collecting terms:

$$(B)\ \ \ \ d_L(x_1,x_2)\leq (4Q+   k_2 k_1^{-{1\over 2}}e^{9\delta+4} )\sqrt{d_V(x_1,x_2)}$$
Let 
$$N=max\{4Q+k_2k_1^{-{1\over 2}}e^{9\delta+4},\ 2(2Q+M)\}$$
Then in all cases, $d_L(x_1,x_2)\leq N(d_V(x_1,x_2))^{1\over 4}$. 
Now given $\epsilon>0$ let $\delta'=({\epsilon\over N})^4$. If $x_1,x_2\in \partial X$ such that $d_V(x_1,x_2)<\delta'$, then $d_L(x_1,x_2)<\epsilon$. Of course we may assume that $\delta'<1$ 
(so that our assumption that $\sqrt {d_V(x_1,x_2)} <d_V(x_1,x_2)^{1\over 4}$ remains valid) must also have that $\delta'$ is small enough to ensure that $m={(x_1.x_2)_\ast\over 2}>57\delta+24+J_{\ref{Deep}}(40\delta+20)$ (so that Lemmas \ref{sm}-\ref{cl2} remain valid).  Note that $(x_1.x_2)_\ast >2(57\delta+24+J_{\ref{Deep}}(40\delta+20))$ if and only if $k_1e^{-(x_1.x_2)_\ast}<k_1e^{-2(57\delta+24+J_{\ref{Deep}}(40\delta+20))}$. As $k_1e^{-(x_1.x_2)_\ast}\leq d_V(x_1,x_2)$, we only need require that $d_V(x_1,x_2)<k_1e^{-2(57\delta+24+J_{\ref{Deep}}(40\delta+20))}$ (equivalently that $\delta'<k_1e^{-2(57\delta+24+J_{\ref{Deep}}(40\delta+20))}$).

The identity function from the compact metric space $(X,d_V)$  to the metric space $(X,d_L)$ is continuous and these metrics induce the same topology.
\end{proof}

\section{The Piecewise Visual Metric is Linearly Connected}\label{Proof}

By definition, $d_L$  agrees with $d_V$ on limit sets of the cosets of the vertex groups. All that remains in order to prove the main theorem of the paper is to combine the results of the previous sections to prove that the metric $d_L$ on $\partial X$ is linearly connected. 

\begin{proof} {\bf (Of Theorem \ref{main})}
Let $x_1\ne x_2$ be points in $\partial X$. Let $C(x_1,x_2)=\{\ldots, c_{-1},c_0,c_1\ldots\}$ be the set of cut points in $\partial X$ separating $x$ and $y$. (We consider the case that $C(x_1,x_2)$ is bi-infinite since the other cases are similar and less complicated.)
Let $g_iV_i$ be the coset of the vertex group $V_i$ of our decomposition of $G$ such that $\{c_i,c_{i+1}\}$ is a subset of the limit set of $X_i(\subset X)$ the cusped space for $g_iV_i$.  By Theorem \ref{sub}, there is a connected set $Q_i$ in $\partial X_i$ containing $c_i$ and $c_{i+1}$ such that $D_V(Q_i)$, the diameter of $Q_i$ under the metric $d_V$, is $\leq q_id_V(c_i,c_{i+1})$ for a constant $q_i$. Since $d_V$ and $d_L$ agree on the limit set of $X_i$, $D_V(Q_i)=D_L(Q_i)$. There are only finitely many vertex groups $V_i$ in our decomposition of $G$, and only finitely many distinct $q_i$ by Lemma \ref{LinCon}.  Let $q$ be the largest of the $q_i$ (over all cosets of vertex groups). The set $\cup_{i=-\infty}^\infty Q_i$ is connected and $x_1,x_2$ are limit points of this set. Hence $Q(x_1,x_2)=\{x_1,x_2\}\cup \cup_{i=-\infty}^\infty Q_i$ is a connected set.
Then 
$$D_L(Q(x_1,x_2))=\sum_{i=-\infty}^\infty D_L(Q_i)\leq q\sum_{i=-\infty}^\infty d_L(c_i,c_{i+1})=qd_L (x_i,x_{i+1})$$ 
\end{proof}
 
\section{The Doubling Question}\label{double}

\begin{definition} 
A metric on a space $(X,d)$ is $n$-doubling if every ball of radius $r$ can be covered by $n$ balls of radius ${r\over 2}$.
\end{definition}

\begin{proposition} \label{doub} 
(\cite {McS18}, Proposition 4.5) The boundary of a relatively hyperbolic group is doubling if and only if each peripheral subgroup is virtually nilpotent.
\end{proposition}

\begin{question} 
If $(G,\mathcal P)$ satisfies the hypotheses of Theorem \ref{main} and each element of $\mathcal P$ is virtually nilpotent, then the visual metric $d_V$ is doubling by Proposition \ref{doub}. 
Is our piecewise visual metric $d_L$ doubling?
\end{question}

For $x_1$ and $x_2$ in $\partial (G,\mathcal P)$ we have shown: 
$d_V(x_1,x_2)<({\epsilon\over N})^4$, implies $d_L(x_1,x_2)<\epsilon$ (where $N$ is a large number) and certainly $d_V(x_1,x_2)\leq d_L(x_1,x_2)$. This implies:
$$B_V(a, ({\epsilon\over N})^4) \subset B_L(a,\epsilon)\subset B_V(a,\epsilon)\subset B_L(a,\epsilon^{1\over 4}N). $$

\bibliographystyle{amsalpha}
\bibliography{paper}{}

\newcommand{\etalchar}[1]{$^{#1}$}
\def\cprime{$'$}
\providecommand{\bysame}{\leavevmode\hbox to3em{\hrulefill}\thinspace}
\providecommand{\MR}{\relax\ifhmode\unskip\space\fi MR }
\providecommand{\MRhref}[2]{%
  \href{http://www.ams.org/mathscinet-getitem?mr=#1}{#2}
}
\providecommand{\href}[2]{#2}
\begin{thebibliography}{ABC{\etalchar{+}}91}

\bibitem[ABC{\etalchar{+}}91]{ABC91}
J.~M. Alonso, T.~Brady, D.~Cooper, V.~Ferlini, M.~Lustig, M.~Mihalik,
  M.~Shapiro, and H.~Short, \emph{Notes on word hyperbolic groups}, Group
  theory from a geometrical viewpoint ({T}rieste, 1990), World Sci. Publ.,
  River Edge, NJ, 1991, Edited by Short, pp.~3--63. \MR{1170363}

\bibitem[BH99]{BrHa99}
Martin~R. Bridson and Andr{\'e} Haefliger, \emph{Metric spaces of non-positive
  curvature}, Grundlehren der Mathematischen Wissenschaften [Fundamental
  Principles of Mathematical Sciences], vol. 319, Springer-Verlag, Berlin,
  1999. \MR{1744486 (2000k:53038)}

\bibitem[BK05]{BK05}
Mario Bonk and Bruce Kleiner, \emph{Quasi-hyperbolic planes in hyperbolic
  groups}, Proc. Amer. Math. Soc. \textbf{133} (2005), no.~9, 2491--2494.
  \MR{2146190}

\bibitem[Bow01]{Bow01}
B.~H. Bowditch, \emph{Peripheral splittings of groups}, Trans. Amer. Math. Soc.
  \textbf{353} (2001), no.~10, 4057--4082. \MR{1837220}

\bibitem[BS07]{BS07}
Sergei Buyalo and Viktor Schroeder, \emph{Elements of asymptotic geometry}, EMS
  Monographs in Mathematics, European Mathematical Society (EMS), Z\"{u}rich,
  2007. \MR{2327160}

\bibitem[GHM{\etalchar{+}}]{GHM19}
Daniel Groves, Peter Ha\"issinsky, Jason Manning, Damion Osajda, Alessandro
  Sisto, and Genevieve Walsh, \emph{Drilling hyperbolic groups}, Preprint.

\bibitem[GM08]{GMa08}
Daniel Groves and Jason~Fox Manning, \emph{Dehn filling in relatively
  hyperbolic groups}, Israel J. Math. \textbf{168} (2008), 317--429.
  \MR{2448064}

\bibitem[MSa]{McS18}
John Mackay and Alessandro Sisto, \emph{Quasi-hyperbolic planes in relatively
  hyperbolic groups}, ArXiv: 1111.2499 [math.{GR}].

\bibitem[MSb]{MS18}
Michael~L. Mihalik and Eric Swenson, \emph{Relatively hyperbolic groups with
  semistable fundamental group at infinity}, ArXiv: 1709.02420 [math.GR].

\end{thebibliography}

 \end{document}